\documentclass[12pt]{article}
\usepackage{amsfonts,latexsym,amsmath,amssymb,amsthm}
\usepackage{mathrsfs,MnSymbol}
\usepackage{url,color}
\usepackage{upgreek}
\usepackage{fancyhdr}
\usepackage{hyperref}
\usepackage{times}
\usepackage{scalefnt}
\usepackage{url}
\usepackage{cite}

\renewcommand{\Re}{\operatorname{Re}}
\renewcommand{\Im}{\operatorname{Im}}

\newcommand{\ZZ}{\mathbb Z}

\newcommand{\RR}{\mathbb R}

\newtheorem{theorem}{Theorem}[section]
\newtheorem{lemma}[theorem]{Lemma}
\newtheorem{proposition}[theorem]{Proposition}
\newtheorem{corollary}[theorem]{Corollary}
\newtheorem{definition}[theorem]{Definition}
\newtheorem{remark}[theorem]{Remark}

\numberwithin{equation}{section}

\numberwithin{equation}{section}
\topmargin - .23 in
\title{Global well-posedness and scattering of the two dimensional cubic focusing nonlinear Schr\"odinger system  }
\author{Xing Cheng$^{*}$, Zihua Guo$^{**}$, Gyeongha Hwang$^{***}$, and Haewon Yoon$^{****}$}

\begin{document}

\maketitle

\renewcommand{\thefootnote}{\fnsymbol{footnote}}
\footnotetext{\hspace*{-5mm}
\begin{tabular}{@{}r@{}p{17cm}@{}}
$^*$ & College of Science, Hohai University, Nanjing 210098, Jiangsu, China. \texttt{chengx@hhu.edu.cn}\\
$^{**}$ & School of Mathematical Sciences, Monash University,  VIC 3800,  Australia.   \texttt{Zihua.Guo@monash.edu} \\
$^{***}$ & Department of Mathematics, Yeungnam University, 280 Daehak-Ro, Gyeongsan, Gyeongbuk 38541, Republic of Korea. \texttt{ghhwang@yu.ac.kr} \\
$^{****}$ & Department of Mathematics, Chung-Ang University, Seoul 06974, Korea. \texttt{hwyoon@cau.ac.kr}
\end{tabular}}

\begin{abstract}
In this article, we prove the global well-posedness and scattering of the cubic focusing infinite coupled nonlinear Schr\"odinger
system on $\mathbb{R}^2$  below the threshold in $L_x^2h^1(\mathbb{R}^2\times \ZZ)$. We first establish the variational characterization of the ground state, and derive the threshold of the global well-posedness and scattering.
 Then we show the global well-posedness and scattering below the threshold by the concentration-compactness/rigidity method,
 where the almost periodic solution is excluded by adapting the argument in the proof of the mass-critical nonlinear Schr\"odinger equations by B. Dodson.
As a byproduct of the scattering of the cubic focusing infinite coupled nonlinear Sch\"odinger system, we obtain the scattering of the cubic focusing nonlinear Schr\"odinger equation on the small cylinder, this is the first large data scattering result of the focusing nonlinear Schr\"odinger equations on the cylinders. In the article, we also show the global well-posedness and scattering of the two dimensional $N-$coupled focusing cubic nonlinear Schr\"odinger system in $\left(L^2(\mathbb{R}^2) \right)^N$.

\bigskip

\noindent \textbf{Keywords}: Nonlinear Schr\"odinger system, ground state, scattering, almost periodic solution, focusing.
\bigskip

\noindent \textbf{Mathematics Subject Classification (2010)} Primary: 35Q55; Secondary: 35P25, 58J37

\end{abstract}

\setcounter{tocdepth}{2}

\section{Introduction}
In this article, we consider the cubic focusing nonlinear Schr\"odinger system on $\mathbb{R}^2$:
\begin{equation}\label{eq1.1}
\begin{cases}
i\partial_t \vec{u}   + \Delta_{\mathbb{R}^2 } \vec{u}  = - \vec{F} (\vec{u}) ,\\
\vec{u}(0) = \vec{u}_{0},
\end{cases}
\end{equation}
 where $\vec{u} = \{ u_j:\mathbb{R}\times\mathbb{R}^2\to \mathbb{C}\}_{j\in \mathbb{Z}_N }$, $\vec{u}_0 = \left\{u_{0,j}:\mathbb{R}^2\to \mathbb{C}\ \right\}_{j\in \mathbb{Z}_N }$, and the nonlinear term 
  $\vec{F} \left(\vec{u} \right) = \left\{\vec{F}_j(\vec{u}) \right\}_{j \in \mathbb{Z}_N }$,
  	\begin{equation}\label{added1}
	\vec{F}_j(\vec{u}) := \sum_{( j_1,j_2,j_3) \in \mathcal{R}(j) } u_{j_1} \bar{u}_{j_2} u_{j_3} =  2\bigg(\sum_{k \in \mathbb{Z}_N } |u_k|^2 \bigg)  u_j - |u_j|^2 u_j ,
	\end{equation}
with
\begin{align*}
\mathcal{R}(j) =
 \left\{ (j_1,j_2,j_3) \in \mathbb{Z}_N^3: j_1-j_2+j_3= j, \, j_1^2 - j_2^2 + j_3^2 = j^2\right\}.
\end{align*}
Here $\mathbb{Z}_N := \{0, 1, \cdots, N-1 \}$, for any $N \in \mathbb{N}_+$, with the convention that $\mathbb{Z}_\infty = \mathbb{Z}$.

The nonlinear Schr\"odinger system \eqref{eq1.1} enjoys the following conservation laws:
\begin{align*}
\text{ \footnotemark   mass:   }     \mathcal{M}_{a,b,c}(\vec{u}(t)) &  = \int_{\mathbb{R}^2  } \sum_{j\in \mathbb{Z}} \left(a+ bj +  c j^2\right) \left|u_j(t,x )\right|^2\,\mathrm{d}x, \text{ where } a,b,c \in \mathbb{R},\\
\intertext{ and }
\text{   energy:  }       \quad \mathcal{E}(\vec{u}(t))  &   =     \int_{\mathbb{R}^2  } \sum_{j\in \mathbb{Z}} \left( \frac12 \left|\nabla u_j(t,x)\right|^2  -  \frac14\left(\bar{u}_j \vec{F}_j(\vec{u})\right) (t,x) \right)   \,\mathrm{d}x\\
& =    \int_{\mathbb{R}^2   } \sum_{j\in \mathbb{Z}}\frac12 \left|\nabla u_j(t,x)\right|^2  - \frac14  \sum_{ \substack{ j\in \mathbb{Z},\\
 n\in \mathbb{N}} }
 \Big|\sum_{\substack{j_1-j_2   =j,\\  j_1^2 - j_2^2   = n}} (  u_{j_1} \bar{u}_{j_2} )(t,x) \Big|^2 \,\mathrm{d}x.
 \end{align*}
\footnotetext{ In fact, $\| \vec{ u} \|_{L_x^2 h^k}$ is conserved for any $k \in \mathbb{N}$,
and therefore the nonlinear Schr\"odinger system \eqref{eq1.1} has infinite conservation laws, which is very interesting. This is pointed out by S. Kwon. }
For $N \in \mathbb{N}_+ $, the nonlinear Schr\"odinger system \eqref{eq1.1} is exactly the $N-$coupled nonlinear Schr\"odinger system:
\begin{align}\label{eq6.1v83}
\begin{cases}
i \partial_t u_j + \Delta u_j = -    |u_j|^2 u_j - 2 \sum\limits_{k \ne j}   |u_k|^2 u_j, \\
u_j(0,x) = u_{0,j}(x), \ j = 0,1,\cdots, N-1,
\end{cases}
\end{align}
where $u_j : \mathbb{R}\times \mathbb{R}^2 \to \mathbb{C}$ for $j=0,1, \cdots, N-1$ and $N\in \mathbb{N}_+$.
 This kind of finite coupled nonlinear Schr\"odinger system has applications in nonlinear optics, see \cite{AA} and the references therein. It is a good approximation describing the propagation of self-trapped mutually incoherent wave packets in nonlinear optics, with $u_j$($j=0,1, \cdots, N-1$) denotes the $j-$th component of the beam.
It also has application in the Bose-Einstein condensates, see \cite{QZL,ZL} and the references therein. We also refer to \cite{X} for the study of other type coupled nonlinear Schr\"odinger system. There are many interesting research works on the coupled nonlinear Schr\"odinger system, especially the ground states, we refer to \cite{LW,WY}.

On the other hand, for the infinite coupled nonlinear Schr\"odinger system, that is $N = \infty$, the nonlinear Schr\"odinger system appears in the nonlinear approximate of the cubic focusing nonlinear Schr\"odinger equations on the cylinder $\mathbb{R}^2\times \mathbb{T}$ in \cite{CGYZ}, where it is called the resonant nonlinear Schr\"odinger system therein. The equation in the defocusing case was studied in \cite{YZ}, a similar nonlinear Schr\"odinger system is derived also in the study of the nonlinear Schr\"odinger equation with partial harmonic potentials in \cite{CGGLS}. We also refer to \cite{HP0} for a different view point.
Although there are a lot of works on the global well-posedness and scattering of the defocusing nonlinear Schr\"odinger equations
on the cylinder $\mathbb{R}^2\times \mathbb{T}$ (See \cite{CGYZ,CGZ,TV2}), there are very few result on the long time behavior of the solutions of the
focusing nonlinear Schr\"odinger equations on the cylinders. We refer to \cite{TTV} for a result on the orbital stability in the focusing case, and global well-posedness result in \cite{YYZ}.

There are some work on other type nonlinear Schr\"odinger systems. We refer to the work of T. Chen, Y. Hong, and N. Pavlovi\'c \cite{CHP,CHP2}. They studied a infinite system of cubic nonlinear Schr\"odinger equations in $\mathbb{R}^d$ when $d\ge 2$, which arises in the mean field quantum fluctuation dynamics for a system of infinitely many fermions with delta pair interactions in the vicinity of an equilibrium solution at zero temperature. In \cite{CHP}, they proved the global well-posedness of the infinite system of cubic nonlinear Schr\"odinger equations in $d=2$ or $3$. Later, they studied the dynamics of the system in $d\ge 3$ near thermal equilibrium and proved scattering in the case of small perturbation around equilibrium in a certain generalized Sobolev space of density operators \cite{CHP2}. The last author together with Y. Hong and S. Kwon \cite{HKY} constructed an extremizer for the Lieb-Thirring energy inequality by developing the concentration-compactness technique for operator valued inequality, and gave the global well-posedness versus finite time blowup dichotomy for the infinite system of focusing cubic nonlinear Schr\"odinger equations in $\mathbb{R}^3$ when each wave function is restricted to be orthogonal.

In the same time, there are some progress on the study of the quadratic nonlinear Schr\"odinger system in $\mathbb{R}^d$, which is mass-critical when $d = 4$.
In \cite{HOT}, N. Hayashi, T. Ozawa, and K. Tanaka studied the equation in $d \leq 6$, and proved the existence of ground states.
In \cite{IKN}, T. Inui, N. Kishimoto, and K. Nishimura studied scattering problem of the quadratic nonlinear Schr\"odinger system in $\mathbb{R}^4$
when the equation satisfies mass-resonance condition or does not satisfy the mass-resonance condition. 
{In \cite{IKN1}, they studied the equation in $d\le 6$, and proved finite time blow-up in the radial case in $d = 5,6$ when the equation does not satisfy mass-resonance condition and prove blow-up or grow-up in $d = 4$.} In \cite{HIN}, M. Hamano, T. Inui, and K. Nishimura proved scattering below the standing wave solution in the radial case when the equation does not satisfy mass-resonance condition.

In this article, we will mainly study the long time behavior of the solution to the nonlinear Schr\"odinger system \eqref{eq1.1}. Because the behavior of the solution to the nonlinear coupled Schr\"odinger system is slight different between $N= \infty $ and $N < \infty$, we divide the discussion of the main result into two subsections.
\subsection{Infinite coupled nonlinear Schr\"odinger system}

We now present the global well-posedness and scattering of the cubic focusing infinite coupled
nonlinear Schr\"odinger system in $L_x^2 h^1(\mathbb{R}^2 \times \mathbb{Z})$.
\begin{theorem}[Global well-posedness and scattering of the cubic focusing infinite coupled nonlinear Schr\"odinger system]\label{th1.2}
For any initial data $\vec{u}_0 \in L_x^2 h^1$ satisfying $\left\|\vec{u}_0\right\|_{L_x^2 l^2} < \frac1{\sqrt2}  \|Q\|_{L_x^2}$, where $Q$ is the ground state of $\Delta_{\mathbb{R}^2} Q - Q = - Q^3$, there exists a global solution $\vec{u} = \left\{u_j\right\}_{j\in \mathbb{Z}}$ to \eqref{eq1.1}, satisfying
\begin{equation*}
\left\|\vec{u}\right\|_{L_{t,x}^4 h^1(\mathbb{R}\times \mathbb{R}^2  \times \mathbb{Z})}  \le C,
\end{equation*}
for some constant $C $ depends only on $\left\|\vec{u}_0\right\|_{L_x^2 h^1}$.
Furthermore, the solution scatters in $L_x^2 h^1$ in the sense that there exists $\left\{u_j^{\pm }\right\}_{j \in \mathbb{Z}} \in  L^2_x h^1$
such that
\begin{equation*}
\left\| \left( \sum\limits_{j\in \mathbb{Z}} \langle j\rangle^2 \left|  u_j(t) - e^{it\Delta_{\mathbb{R}^2}} u_j^{\pm  }\right|^2 \right)^\frac12  \right\|_{L^2(\mathbb{R}^2 )} \to 0, \text{ as } t\to \pm \infty.
\end{equation*}
\end{theorem}

\begin{remark}
The threshold of scattering is sharp in the sense that if the mass is greater than $\frac12 \|Q\|_{L^2}^2$, we have finite time blow up.
We believe when the mass is equal to the threshold, the solution still scatters, we refer to the similar results on the 2-D cubic-quintic NLS, see \cite{CS,Ch,Mu}. This will be discussed in our future project, \cite{CGHY}.

\end{remark}

As a consequence of Theorem \ref{th1.2}, by the argument in the proof of Theorem 3.9 in \cite{CGYZ}, we can obtain the global well-posedness and scattering of the large-scale solution of the focusing cubic NLS on $\mathbb{R}^2 \times \mathbb{T}$ in $L_x^2 H_y^1$, where $\mathbb{T} = \mathbb{R}/2\pi \mathbb{Z}$.
\begin{theorem}[GWP \& scattering of the large-scale solution of the focusing cubic NLS on the cylinder]
\label{le3.11v63}
Let $\phi \in L^2_x H_y^1(\mathbb{R}^2\times \mathbb{T})$ with $\| \phi \|_{L_{x,y}^2 } < \frac1{\sqrt2}  \|Q\|_{L^2}$
be given, then there is $\lambda_0=  \lambda_0 (\phi)$ sufficiently large such that for $\lambda \ge \lambda_0$, we have a unique global solution
$U_\lambda \in C_t^0 L_x^2 H_y^1(\mathbb{R} \times \mathbb{R}^2 \times \mathbb{T})$ of
\begin{equation}\label{eq1.3v72}
\begin{cases}
i \partial_t U_\lambda + \Delta_{\mathbb{R}^2 \times \mathbb{T}} U_\lambda = - \left|U_\lambda\right|^2 U_\lambda,
\\
U_\lambda(0,x,y) = \frac1\lambda \phi\left( \frac{x}\lambda, y \right).
\end{cases}
\end{equation}
Moreover, for $\lambda \ge \lambda_0$, we have
\begin{align*}
\left\|U_\lambda \right\|_{L_t^\infty L_x^2 H_y^1 \cap L_{t,x}^4 H_y^1(\mathbb{R} \times \mathbb{R}^2 \times \mathbb{T})} \lesssim_{\|\phi\|_{L_x^2 H_y^1}} 1.
\end{align*}
As a consequence, $U_\lambda$ scatters in $L_x^2 H_y^1$ in the sense that there exist $ \left\{ U_\lambda^\pm \right\} \in L_x^2 H_y^1 $ such that
\begin{align*}
\left\| U_\lambda (t) - e^{it \Delta_{\mathbb{R}^2 \times \mathbb{T}} } U_\lambda^\pm \right\|_{L_x^2 H_y^1 } \to 0, \text{ as } t \to \pm \infty.
\end{align*}
\end{theorem}

Let
\begin{align*}
\tilde{U}_\lambda (t,x,y) =  \lambda  U_\lambda ( \lambda^2 t , \lambda x, y),
\end{align*}
we have
\begin{align*}
\begin{cases}
i \partial_t \tilde{U}_\lambda + \Delta_{\mathbb{R}^2 \times { \mathbb{T}_{ \lambda^{-1} } } } \tilde{U}_\lambda = - \left|\tilde{U}_\lambda \right|^2 \tilde{U}_\lambda, \\
\quad \\
\tilde{U}_\lambda(0,x,y)  = \tilde{\phi}(x,y),
\end{cases}
\end{align*}
where $\tilde{\phi}$ is a modification of the periodic function $\phi$ with respect to $y$ on $ [ -\pi, \pi]$, such that
$\tilde{\phi}(x,y)$ is periodic with respect to $y$ on  $\mathbb{T}_{\lambda^{-1}} : = \lambda^{-1} \mathbb{T} = [ - \lambda^{-1} \pi, \lambda^{-1} \pi ]$, and $\tilde{\phi} = \phi, a.e.$

Thus, we get the global well-posedness and scattering in $L_x^2 H_y^1(\mathbb{R}^2 \times \mathbb{T}_{\lambda^{-1} } )$ of the focusing cubic NLS on the small cylinder $\mathbb{R}^2 \times  \mathbb{T}_{ \lambda^{-1} } $, where $\lambda$ is sufficient large constant.
\begin{theorem}[GWP \& scattering of the focusing cubic NLS on the small cylinder]\label{le3.11v63}
Let $\phi \in L^2_x H_y^1(\mathbb{R}^2\times \mathbb{T})$ with $\| \phi \|_{L_{x,y}^2 } < \frac1{\sqrt2}  \|Q\|_{L^2}$
be given, then there is $\lambda_0=  \lambda_0 (\phi)$ sufficiently large such that for $\lambda \ge \lambda_0$, we have a unique global solution
$U_\lambda \in C_t^0 L_x^2 H_y^1(\mathbb{R} \times \mathbb{R}^2 \times  \mathbb{T}_{ \lambda^{-1}})$ of
\begin{equation}\label{eq1.3v72}
i \partial_t U_\lambda + \Delta_{\mathbb{R}^2 \times  \mathbb{T}_{ \lambda^{-1} }} U_\lambda = -  \left|U_\lambda \right|^2 U_\lambda,
\end{equation}
with $U_\lambda(0,x,y) = \tilde{  \phi} ( {x}, y), \text{ where $\tilde{\phi}$ is a periodic modification of $\phi$}$.
Moreover, for $\lambda \ge \lambda_0$,
\begin{align*}
\|U_\lambda \|_{L_t^\infty L_x^2 H_y^1 \cap L_{t,x}^4 H_y^1(\mathbb{R} \times \mathbb{R}^2 \times  \mathbb{T}_{  \lambda^{-1} })} \lesssim_{\|\phi\|_{L_x^2 H_y^1}} 1.
\end{align*}
As a consequence, $U_\lambda$ scatters to the solution of the linear equation $ i \partial_t V_\lambda + \Delta_{\mathbb{R}^2 \times  \mathbb{T}_{ \lambda^{-1} } } V_\lambda = 0$ in $L_x^2 H_y^1(\mathbb{R}^2 \times  \mathbb{T}_{ \lambda^{-1}}) $ when $\lambda$ is sufficiently large.

\end{theorem}

\begin{remark}
In the above theorem, we only get the global well-posedness and scattering of the cubic focusing nonlinear Schr\"odinger equation on the small cylinder. It seems difficult to give the global well-posedness and scattering of the focusing nonlinear Schr\"odinger equation on the general cylinder. The main obstacle is a lack of knowledge on the threshold of the global well-posedness, which is closely related to the sharp Gagliardo-Nirenberg inequality on the cylinder.

\end{remark}

\subsection{Finite couple nonlinear Schr\"odinger system}

The argument in the proof of Theorem \ref{th1.2} also works for the $N-$coupled cubic focusing nonlinear Schr\"odinger system,
where $N$ is any positive integer.

Let
\begin{align*}
 \left(L^2(\mathbb{R}^2) \right)^N : = \overbrace{L^2(\mathbb{R}^2) \times L^2(\mathbb{R}^2) \times \cdots \times L^2(\mathbb{R}^2)}^{N-\text{copy}} ,
\end{align*}
we have
\begin{theorem}[GWP \& scattering of the cubic focusing $N-$coupled nonlinear Schr\"odinger system]\label{th1.3}
For any initial data $\vec{u}_0 \in \left( L_x^2 \right)^N $ satisfying $\left\|\vec{u}_0\right\|_{ \left( L_x^2 \right)^N  }   <  \sqrt{ \frac{N}{2N- 1 } }   \, \|Q\|_{L_x^2} $,
 where $Q$ is the ground state of $\Delta_{\mathbb{R}^2} Q - Q = - Q^3$, there exists a global solution
 $\vec{u} = \left\{u_j\right\}_{j\in \mathbb{Z}_N }$ to \eqref{eq1.1} satisfying
\begin{equation*}
\left\|\vec{u}\right\|_{L_{t,x}^4  (\mathbb{R}\times \mathbb{R}^2   )}  \le C,
\end{equation*}
for some constant $C $ depends only on $\left\|\vec{u}_0\right\|_{ \left( L_x^2  \right)^N  }$.
In addition, the solution scatters in $ \left( L_x^2  \right)^N  $ in the sense that there exists $\left\{u_j^{\pm }\right\}_{j \in \mathbb{Z}_N } \in   \left( L^2_x  \right)^N  $
such that
\begin{equation*}
\left\| \left( \sum\limits_{j\in \mathbb{Z}_N  }  \left|  u_j(t) - e^{it\Delta_{\mathbb{R}^2}} u_j^{\pm  }\right|^2 \right)^\frac12  \right\|_{L^2(\mathbb{R}^2 )} \to 0, \text{ as } t\to \pm \infty.
\end{equation*}
\end{theorem}

\begin{remark}
When  $\left\|\vec{u}(0) \right\|_{L^2}^2 =  \frac{N}{2N-1}
 \left\| {Q}\right\|_{L^2}^2 $, we expect a rigidity theorem, and refer to the recent work on the mass-critical NLS when the mass is equal to the threshold of scattering, see \cite{D5,D7,F}.

\end{remark}

When $N= 2  $, the system  \eqref{eq1.1} degenerates into the $2-$coupled nonlinear Schr\"odinger system:
\begin{equation}\label{eq1.6}
\begin{cases}
i\partial_t u_0 + \Delta_{\mathbb{R}^2} u_0 = -  |u_0|^2 u_0 -   2|u_1|^2 u_0,\\
i\partial_t u_1 + \Delta_{\mathbb{R}^2} u_1 = - 2|u_0|^2 u_1 - |u_1|^2 u_1,
\end{cases}
\end{equation}
with $u_j(0,x) = u_{0,j}(x)$ for $j=0,1$. \eqref{eq1.6} is also the non-relativistic limit of the complex-valued cubic focusing nonlinear Klein-Gordon equation in $\mathbb{R}^2$, we refer \cite{CGM,MN}.

As a straightforward consequence of Theorem \ref{th1.3}, we have
\begin{corollary}[Global well-posedness and scattering of the $2-$coupled focusing nonlinear Schr\"odinger system]
\eqref{eq1.6} is globally well-posed for $\vec{u}(0) \in L^2 \times L^2$ satisfying $\left\|\vec{u}(0) \right\|_{L^2}^2 < \frac23 \left\| {Q}\right\|_{L^2}^2 $, and $\vec{u}(t)$ scatters to a free solution as $t \to \pm \infty$.
\end{corollary}

In this article, we first establish the variational characteristic of the corresponding ground state of \eqref{eq1.1} for both $N = \infty$ and $N< \infty$. First, we consider the focusing mass-critical nonlinear Schr\"odinger equation
\begin{align*}
i\partial_t u + \Delta_{\mathbb{R}^2} u = - |u|^2 u.
\end{align*}
The corresponding ground state is given by a solution to the following elliptic equation
\begin{align*}
\Delta_{\mathbb{R}^2} Q  - Q = - Q^3
\end{align*}
which is known to be unique up to modulo space translations and multiplication by $e^{i \theta}$, see \cite{Kw}. Hence the threshold of the scattering is also unique \cite{W}.
However, for \eqref{eq1.1}, there are many positive solutions up to modulo space translations and multiplication by $e^{i \theta}$ of the corresponding elliptic equation
\begin{align*}
\Delta_{\mathbb{R}^2} Q_j - Q_j = - \sum\limits_{(j_1,j_2,j_3) \in \mathcal{R}(j)} Q_{j_1} Q_{j_2} Q_{j_3}.
\end{align*}
We refer to  \cite{CW,I,LW,Si,WY} for the finite coupled case, this reflects the complex behaviour of the nonlinear Schr\"odinger system. However, we find the best constant of the sharp Gagliardo-Nirenberg
inequality is unique and thus the threshold of the scattering is also unique, which is $\frac12 \|Q\|_{L_x^2}^2$.
\begin{remark}\label{re1.5v85}
We have
\begin{align*}
\partial_t \left( \int |x|^2 \sum\limits_{j \in \mathbb{Z}} |u_j(t,x)|^2 \,\mathrm{d}x\right)
= 4 \sum\limits_{j \in \mathbb{Z}} \int x \cdot \Im \left( \bar{u}_j \nabla u_j \right) \,\mathrm{d}x
\end{align*}
and
\begin{align*}
\partial_t  \left(  4 \sum\limits_{j \in \mathbb{Z}} \int x \cdot \Im \left( \bar{u}_j \nabla u_j \right) \,\mathrm{d}x \right) = 16 E(\vec{u}).
\end{align*}
Therefore, when $ \left\|\vec{u} \right\|_{L_x^2 l^2}^2 < \frac12 \|Q\|_{L_x^2}^2$, we see
\begin{align*}
&\quad \partial_t^2 \left( \int |x|^2 \sum\limits_{j \in \mathbb{Z}} |u_j(t,x)|^2 \,\mathrm{d}x\right) =
\partial_t  \left(  4 \sum\limits_{j \in \mathbb{Z}} \int x \cdot \Im \left( \bar{u}_j \nabla u_j \right) \,\mathrm{d}x \right)\\
& = 16 E(\vec{u})
= 8 \left\|\nabla \vec{u} \right\|_{L_x^2 l^2}^2 - 4 \mathcal{N}( \vec{u})
 \ge 8 \left( \left\|\nabla \vec{u} \right\|_{L_x^2 l^2}^2 - \frac{ \left\|\vec{u} \right\|_{L_x^2 l^2}^2 }{ \frac12 \|Q\|_{L_x^2}^2}  \left\|\nabla \vec{u} \right\|_{L_x^2 l^2}^2\right) > 0.
\end{align*}

On the other hand, for any $\epsilon > 0$, there exists $\vec{u}(0)$ which satisfies $\|\vec{u}(0)\|_{L^2_x  l^2} = \frac12\|Q\|_{L^2} + \epsilon$, with $E(\vec{u}(0)) < 0$, $\int |x|^2 \| \vec{u}(0,x)\|_{l^2}^2 \,\mathrm{d}x < \infty$, and $ \sum\limits_{j \in \mathbb{Z}} \int  x\cdot \Im ( \bar{u}_j(t,x) \nabla_x u_j(t,x)) \,\mathrm{d}x < \infty$.
Then the solution of \eqref{eq1.1} blows up in finite time for such initial data $\vec{u}(0)$. Thus the threshold of global well-posedness and scattering for \eqref{eq1.1} is sharp.

\end{remark}

Once finding the exact threshold of global well-posednss and scattering, we will show the global well-posedness and scattering below the threshold. We will focus on the cubic focusing infinite coupled nonlinear Schr\"odinger system because this case is more difficult than the cubic focusing finite coupled nonlinear Schr\"odinger system, in fact, the finite case is a trivialization of the infinite case. We now explain the idea of the proof of the global well-posedness and scattering below the threshold. 
To prove the global well-posedness and scattering below the threshold, it suffices to show the finiteness of the $L_{t,x}^4 l^2$ norm of the solution. By using the concentration-compactness/rigidity method, we get an almost periodic solution which is almost periodic module the scaling, Galilean transformation and spatial translation. So we only need to exclude the almost periodic solution. To exclude the almost periodic solution, we divide the almost periodic solution into two different kinds according to whether $\int_0^\infty N(t)^3 \,\mathrm{d}t$ is finite or not. If $\int_0^\infty N(t)^3 \,\mathrm{d}t$ is finite, the almost periodic solution has higher regularity and belongs to $H_x^3 l^2$. So we can exclude the almost periodic solution by exploiting the conservation of energy. Otherwise, i.e. if $\int_0^\infty N(t)^3 \,\mathrm{d} t$ is infinite, the exclusion of the almost periodic solution relies on the argument of B. Dodson \cite{D4}. We first modify the scale function $N(t)$ so that it oscillates less, and then we exclude the solution by exploiting the low frequency localized interaction Morawetz estimate.

The remaining part of the paper is organized as follows. We will focus on the proof of Theorem \ref{th1.2} from Section \ref{se2} to Section \ref{se5v30}, precisely, we give the local well-posedness theory and small data scattering in Section \ref{se2}.
In Section \ref{se3v30}, we establish the variational estimate of the ground state.
{In Section \ref{se4v30}, we reduce the non-scattering to the existence of the almost periodic solution.} In Section \ref{se5v30}, we preclude the almost periodic solution by using the method of B. Dodson developed in \cite{D4}.
In Section \ref{s1se3v30}, we give a sketch of the scattering of the $N-$coupled NLS.

\subsection{Notation and Preliminaries}
We will use the notation $X\lesssim Y$ whenever there exists some constant $C>0$ so that $X \le C Y$. Similarly, we will use $X \sim Y$ if
$X\lesssim Y \lesssim X$.

For the vector-valued function $\vec{f}(t,x)  = \left\{f_j(t,x)  \right\}_{j\in \mathbb{Z} }$, we denote
\begin{align*}
\left\| \vec{f}\right\|_{L_t^p L_x^q h^s } := \bigg \|\Big(\sum_{j \in \mathbb{Z}}  \langle j \rangle^{2s} \left| f_j(t,x) \right|^2 \Big)^\frac12 \bigg\|_{L_t^p L_x^q},
\end{align*}
where $0 \le s \le 1$. When $s = 0$, we write $L_t^p L_x^q h^0$ to be $L_t^p L_x^q l^2$.

We define the discrete nonisotropic Sobolev space. For $\vec{\phi} =  \left\{\phi_k \right\}_{k\in \mathbb{Z}}$ a sequence of real-variable functions, we define
\begin{align*}
 H^{s_1}_x h^{s_2} := \left\{ \vec{\phi} = \left\{ \phi_k\right\}_{k\in \mathbb{Z}} : \left\|\vec{\phi}\right\|_{H_x^{s_1} h^{s_2}}  =   \bigg \| \bigg( \sum\limits_{k\in \mathbb{Z}} \langle k\rangle^{2s_2} \left|\langle \nabla_x\rangle^{ s_1} \phi_k (x)  \right|^2  \bigg)^\frac12 \bigg\|_{L_x^{2}}  < \infty \right \},
\end{align*}
where $s_1,s_2 \ge 0$. In particular, when $s_1 = 0$, we denote the space $H_x^{s_1} h^{s_2}$ to be $L_x^2 h^{s_2}$.

We also define
\begin{align*}
\mathcal{N} \left(\vec{f} \right)  = \int_{\mathbb{R}^2}  N \left(\vec{f} \right) \,\mathrm{d}x
 := \int_{\mathbb{R}^2} \sum\limits_{j \in \mathbb{Z}} \left( \bar{f}_j \vec{F}_j \left(\vec{f} \right)\right)(x)
 \,\mathrm{d}x.
\end{align*}
We also have the following important observation.
\begin{lemma}[The equivalence of space-time norms]
\begin{align}\label{eq1.5v83}
\mathcal{N}(\vec{u}) \sim  \left\|\vec{u} \right\|_{L_{t,x}^4 l^2}^4.
\end{align}
\end{lemma}
\begin{proof}
By the elementary inequality
\begin{align*}
\left( \sum\limits_{j \in \mathbb{Z}} |u_j|^4 \right) \le   \left( \sum\limits_{j \in \mathbb{Z}} |u_j|^2 \right)^2,
\end{align*}
we have
\begin{align*}
\left( \sum\limits_{j \in \mathbb{Z}} |u_j|^2 \right)^2 \le
2 \left( \sum\limits_{j \in \mathbb{Z}} |u_j|^2 \right)^2 - \sum\limits_{j \in \mathbb{Z}} |u_j|^4
\le 3 \left( \sum\limits_{j \in \mathbb{Z}} |u_j|^2 \right)^2.
\end{align*}
Thus
\begin{align*}
\sum\limits_{j \in \mathbb{Z}} \sum\limits_{ (j_1,j_2,j_3) \in \mathcal{R}(j)} \bar{u}_j u_{j_1} \bar{u}_{j_2} u_{j_3} = 2 \left( \sum\limits_{j \in \mathbb{Z}} |u_j|^2 \right)^2 - \sum\limits_{j \in \mathbb{Z}} |u_j|^4
\sim \left( \sum\limits_{j \in \mathbb{Z}} |u_j|^2 \right)^2 \sim \|\vec{u}\|_{l^2}^4.
\end{align*}
Therefore, we get \eqref{eq1.5v83}.

\end{proof}
\begin{lemma}
By Minkowski's inequality, interpolation, and H\"older's inequality, we have
\begin{align}\label{eq5.27v32}
\| \vec{v} \|_{L_x^4 l^2}
 &  \lesssim \left( \sum\limits_{j \in \mathbb{Z}} \|v_j(x) \|_{L_x^2} \|v_j(x) \|_{\dot{H}_x^1} \right)^\frac12
\lesssim 
\|\vec{v}\|_{L_x^2 l^2}^\frac12 \|\vec{v}\|_{\dot{H}^1_x l^2}^\frac12.
\end{align}
\end{lemma}

\section{Local well-posedness and small data scattering}\label{se2}
In this section, we will review the local wellposedness, small data scattering and the stability theory, for detailed exposition, 
 we refer to \cite{CGYZ,CGZ,YZ}.

\begin{definition}[Strichartz admissible pair]
We call a pair $(q,r)$ is Strichartz admissible pair if $2 < q \le \infty$, $2\le r < \infty$, and $\frac1q + \frac1r = \frac12$.
\end{definition}
\begin{theorem}[Strichartz estimate,\cite{CGYZ,TV,TV2,YZ}]\label{th2.1v31}
For any $\alpha = 0,1$,
\begin{align*}
  \left\|e^{it \Delta} \vec{f}\right\|_{L_t^q L_x^r h^\alpha}   \lesssim \left\|\vec{f}\right\|_{L_x^2 h^\alpha}   \text{ and }
 \left\|\int_0^t e^{i(t-s) \Delta} \vec{F}(s,x) \,\mathrm{d}s \right\|_{L_t^q L_x^r h^\alpha} \lesssim \left\|\vec{F}\right\|_{L_t^{\tilde{q}'} L_x^{\tilde{r}'} h^\alpha},
\end{align*}
where $(q,r)$ and $(\tilde{q}, \tilde{r})$ are Strichartz admissible.
\end{theorem}

By standard arguments with Theorem \ref{th2.1v31}, we obtain
\begin{proposition}[Local wellposedness and small data scattering]\label{lea1}
Let $\vec{u}(0)= \{u_p(0)\}_p \in  L_x^2 h^1$ satisfying $\|\vec{u}_0\|_{ L^2_x h^1} \le E$ for some $E > 0$. Then
there exists an open interval $ I \ni 0$ and a unique solution $\vec{u}(t)$ of \eqref{eq1.1} in $C_t^0 L_x^2 h^1(I\times \mathbb{R}^2 \times \mathbb{Z} ) \cap L_{t,x}^4 h^1(I \times \mathbb{R}^2 \times \mathbb{Z} )$.
In addition, if $\vec{u}(0) \in  H_x^k h^\sigma(\RR^2 \times \ZZ)$ for some $\sigma \ge 1$ and $k\ge 0$, then $\vec{u}(t) \in C_t^0 H_x^k h^\sigma(I\times \mathbb{R}^2 \times \mathbb{Z}  )$.
Moreover, there exists $\delta_0 > 0$ such that if $E \le \delta_0$, then $\vec{u}(t)$ is global and scatters in positive and negative infinite time.
\end{proposition}
\begin{theorem}[Scattering norm]\label{th2.4v103} 
If the solution $\vec{u}$ of \eqref{eq1.1} satisfies
\begin{align*}
\left\|\vec{u}\right\|_{L_{t,x}^4 l^2(\mathbb{R}\times \mathbb{R}^2 \times \ZZ ) } < \infty,
\end{align*}
we have scattering in $L_x^2 h^1$, that is there exist $\vec{u}^\pm \in L_x^2 h^1$ such that
\begin{align*}
\left\|\vec{u}(t) - e^{it \Delta} \vec{u}^\pm \right\|_{L_x^2 h^1} \to 0, \text{ as } t \to \pm \infty.
\end{align*}
\end{theorem}
\begin{theorem}[Stability theorem]
For any $\alpha = 0,1$, let $I$ be a compact interval and $\vec{e} = \{e_j\}_{j\in \mathbb{Z}}$,
\begin{align*}
e_j = i \partial_t u_j + \Delta u_j + \vec{F}_j(\vec{u}).
\end{align*}
with
\begin{align*}
\left\|\vec{u}\right\|_{L_{t,x}^4 h^\alpha(I \times \mathbb{R}^2)} \le A,
\end{align*}
for $A> 0$. Then for any $\epsilon > 0$, there is $\delta >0$, such that if
\begin{align*}
\left\|\vec{e}\right\|_{L_{t,x}^\frac43 h^\alpha} \le \delta \text{ and }
\left\|\vec{u}(t_0) - \vec{v}_0 \right\|_{L_x^2 h^1} \le \delta,
\end{align*}
then \eqref{eq1.1} has a solution $\vec{v} \in C_t^0 L_x^2 h^1 \cap L_{t,x}^4 h^\alpha(I \times \mathbb{R}^2\times \ZZ)$
 with initial data $\vec{v}(t_0) = \vec{v}_0$.
Furthermore, 
\begin{align*}
\left\|\vec{u}- \vec{v}\right\|_{L_{t,x}^4 h^\alpha} + \left\|\vec{u}- \vec{v}\right\|_{L_t^\infty L_x^2 h^1} \le \epsilon.
\end{align*}
\end{theorem}

\section{Variational characterization of the ground state}\label{se3v30}
In this section, we study the sharp constant $C_{res}$ of the following Gagliardo-Nireberg inequality:
\begin{align}\label{eq3.1v99}
2\int_{\mathbb{R}^2}\bigg(\sum_{j\in \mathbb{Z}_N } |u_j|^2\bigg)^2- \int_{\mathbb{R}^2}\sum_{j\in  \mathbb{Z}_N }|u_j|^4 \leq C_{res} \bigg(\int_{\mathbb{R}^2}\sum_{j\in  \mathbb{Z}_N } |u_j|^2 \bigg)\bigg(\int_{\mathbb{R}^2} \sum_{j\in \mathbb{Z}_N  } |\nabla u_j|^2 \bigg),
\end{align}
where $\vec{u}\in H_x^1\ell^2_j(\mathbb{R}^2 \times \mathbb{Z} )$ when $N = \infty$ and $\vec{u} \in \left( H^1_x \right)^N $ when $N \in \mathbb{N}$.

When $N= 1$, we have the following know result.
\begin{lemma}[Sharp Gagliardo-Nirenberg inequality for the scalar function]\label{le3.1v86}
\begin{align*}
\|u\|_{L_x^4}^4 \le \frac2{\|Q\|_{L_x^2}^2} \|u\|_{L_x^2}^2 \|\nabla u \|_{L_x^2}^2,
\end{align*}
and $Q$ is the ground state of $\Delta Q - Q = - Q^3$, the equality holds if and only if $u(x) = c \tilde{Q}(x)$, where $c$ is a constant and $\tilde{Q}$ is
$Q$ under the action of the group of translation and dilation.

\end{lemma}

When $1 < N < \infty$, we denote $\vec{u} = (u_0,u_1, \cdots, u_{N-1})$. It is obviously that
$\Big( \overbrace{\sqrt{\frac1{2N-1}} Q, \cdots, \sqrt{\frac1{2N-1}} Q }^{N-\text{ copy}} \Big)$  is a solution of the corresponding elliptic equation system
\begin{equation}\label{eq1.4v9}
 \Delta Q_j - Q_j = -    | Q_j|^2 Q_j - 2  \sum\limits_{k \ne j}   |Q_k|^2 Q_j, \text{ for } j = 0, 1, \cdots, N-1,
\end{equation}
where $Q$ is the ground state solution of
$\Delta_{\mathbb{R}^2} Q - Q = - Q^3$. When $N=2$, B. Sirakov \cite{Si} conjectured that under spatial translation and rotation that $\left(\sqrt{\frac13} Q, \sqrt{\frac13} Q \right)$ is the unique positive solution to \eqref{eq1.4v9}. This conjecture was proved by J. Wei and W. Yao \cite{WY} (see also \cite{I}).
These arguments can be extended to the general finite $N > 1$ case, and thus under spatial translation and rotation that
$\Big( \overbrace{\sqrt{\frac1{2N-1}} Q, \cdots, \sqrt{\frac1{2N-1}} Q }^{N-\text{ copy}}  \Big)$ is the unique positive solution to \eqref{eq1.4v9}.
N. V. Nguyen, R. Tian, B. Deconinck, and N. Sheils \cite{NTDS} showed the ground state is connected with the best constant for the vector valued sharp Gagliardo-Nirenberg inequality for $\vec{f} = (f_0,f_1, \cdots, f_{N-1} )$,
\begin{align}\label{eq1.2v9}
\mathcal{N}(\vec{f})  \le C_{op} \left\|\nabla_x \vec{f} \right\|_{L_x^2(\mathbb{R}^2)}^2 \left\|\vec{f}\right\|_{L_x^2(\mathbb{R}^2) }^2,
\end{align}
where
\begin{align*}
& \left\|\nabla_x \vec{f}\right\|_{L_x^2} = \left( \sum\limits_{j=0}^{N-1} \int_{\mathbb{R}^2} |\nabla f_j( x)   |^2  \,\mathrm{d}x\right)^\frac12, \ 
\left\|\vec{f}\right\|_{L_x^2} = \left( \sum\limits_{j=0}^{N-1}  \int_{\mathbb{R}^2} |f_j(x)  |^2   \,\mathrm{d}x\right)^\frac12, \\
\text{ and } &
\mathcal{N}(\vec{f}) = \sum\limits_{j=0}^{N-1}    \int_{\mathbb{R}^2}  \left( \|f_j(x) \|_{L^4}^4 + 2
 \sum\limits_{ 0 \le k \le N-1, \atop k\ne j } |f_k(x)|^2 |f_j(x)|^2 \right)  \,\mathrm{d}x.
\end{align*}

We will now study the best constant of the Gagliardo-Nirenberg inequality \eqref{eq3.1v99} for $N \in \mathbb{N} \cup \{ \infty \}$. Let $J\subseteq \mathbb{Z}$ be the index set whose size can be finite or (countably) infinite, for any nontrivial $\vec{u} =\{u_j\}_{j\in \mathbb{Z}_N } \subseteq  H_x^1\ell^2(\mathbb{R}^2\times J)$, we define the Weinstein functional
	\[W(\vec{u}) :=	\frac{2\int_{\mathbb{R}^2}\big(\sum\limits_{j\in J} |u_j|^2\big)^2- \int_{\mathbb{R}^2}\sum\limits_{j\in J}|u_j|^4}
	{\big(\int_{\mathbb{R}^2}\sum\limits_{j\in J} |u_j|^2 \big)\big(\int_{\mathbb{R}^2} \sum\limits_{j\in J} |\nabla u_j|^2 \big)},\]
where the size of $J$ can be finite or countably infinite, i.e., $|J|\in\{1,2,\cdots\}\cup\{\infty\}$.
To emphasize the role of the size of $|J|$, we denote $W_{|J|}(\vec{u})=W(\vec{u})$, and
	\[C_{|J|}=\sup\left\{W_{|J|}(\vec{u}) : \vec{u} =    \left\{u_j \right\}_{j\in J} \in   H_x^1\ell^2_j(\mathbb{R}^2\times J),\ \vec{u}\neq \vec0 \right\}.\]
Note that the Weinstein functional $W \left(\vec{u} \right)$ is invariant under homogeneity and scaling symmetry, that is, $W(\vec{u})=W(\vec{u}^{\lambda,\mu})$ where $u^{\lambda,\mu}_j (\,\cdot\,):=\mu u_j \left(\lambda\,\cdot\, \right)$ for any $\mu$, $\lambda>0$ and $j\in J$.
By standard variational argument, a maximizer $\vec{Q}= \left\{Q_j \right\}_{j\in J}$ of the Weinstein functional $W$ weakly solves the system of Euler-Lagrange equations
	\begin{equation}\label{eq:EL}
	\Delta_{\mathbb{R}^2} Q_j - Q_j +  2 \bigg(\sum_{j\in J} |Q_j|^2\bigg) Q_j - |Q_j|^2 Q_j =0,\quad j\in J,
	\end{equation}
if it exists.

\begin{remark}[Properties of maximizer]\label{rem:max}
If there exists a maximizer $\vec Q=\{Q_j\}_{j\in J}$, then it satisfies the following properties:
\begin{enumerate}
\item By switching every $u_j$ by $|u_j|$, we may assume that the maximizer $\vec Q$ is non-negative.
Then by standard argument by maximum principle, each component $Q_j$ of such non-negative maximizer is indeed strictly positive.

\item For $j\in J$, denote $Q_j^*$ as a symmetric decreasing rearrangement (Schwarz symmetrization) \cite{Lieb-Loss} of strictly positive function $Q_j \in H^1(\mathbb{R}^2)$.
Let $\vec Q^*=\{Q_j^*\}_{j\in J}$.
Then, it is well-known that
	\[\int_{\mathbb{R}^2} Q_i^2Q_j^2\:dx\leq \int_{\mathbb{R}^2} (Q_i^*)^2(Q_j^*)^2\:dx,\quad i,\ j\in J.\]
Especially, by summing over $i$, $j\in J$, it follows that
	\[ 2\int_{\mathbb{R}^2}\Big(\sum_{j\in J} Q_j^2\Big)^2- \int_{\mathbb{R}^2}\sum_{j\in J}Q_j^4 \leq 2\int_{\mathbb{R}^2}\Big(\sum_{j\in J} (Q_j^*)^2\Big)^2- \int_{\mathbb{R}^2}\sum_{j\in J}(Q_j^*)^4.\]
Now, we can take every $Q_j$ to be radial, i.e., $Q_j(x)=Q_j(|x|)$, since we get $W(\vec Q)\leq W(\vec Q^*)$ from Polya-Szeg\"{o} inequality.

\item From the simple observation (as in \cite{WY} for instance), by multiplying $i$-th equation of \eqref{eq:EL} by $Q_j$ and vice versa, we have
	\[\int_{\mathbb{R}^2}\nabla Q_i\cdot \nabla Q_j + \int_{\mathbb{R}^2} Q_i Q_j - \int_{\mathbb{R}^2} \Big(\sum_{j\in J}|Q_j|^2\Big)Q_iQ_j + \int_{\mathbb{R}^2}|Q_j|^2Q_iQ_j =0,\]
	\[\int_{\mathbb{R}^2}\nabla Q_i\cdot \nabla Q_j + \int_{\mathbb{R}^2} Q_i Q_j - \int_{\mathbb{R}^2} \Big(\sum_{j\in J}|Q_j|^2\Big)Q_iQ_j + \int_{\mathbb{R}^2}|Q_i|^2Q_iQ_j =0.\]
In particular, these implies that
	\[\int_{\mathbb{R}^2} (Q_i^2-Q_j^2)Q_iQ_j =\int_{\mathbb{R}^2} (Q_i-Q_j)(Q_i+Q_j)Q_iQ_j =0.\]
By strict positivity of each $Q_i$, we have $Q_i=Q_j$.
\end{enumerate}
\end{remark}
When $|J|<\infty$, simply following the compactness argument of Weinstein \cite{W}, the maximizer $\vec{Q} =\{Q_j\}_{j\in J}$ exists.
From the third observation in the Remark \ref{rem:max}, we can let $Q_j=\phi$, and the Euler-Lagrange equation \eqref{eq:EL} will be decoupled and written as single equation $\Delta_{\mathbb{R}^2} \phi - \phi + |\phi|^2 \phi =0$.
Now, we have
\begin{equation}\label{eq:QoptimalC}
C_{|J|}=W_{|J|}(\vec Q)=\frac{2|J|-1}{|J|}\cdot\frac{\|Q\|_{L^4}^4}{\|\nabla Q\|_{L^2}^2\|Q\|_{L^2}^2}=\frac{2 \left(2|J|-1 \right)}{|J|}\cdot\frac{1}{\|Q\|_{L^2}^2},
\end{equation}
where $Q$ is the ground state solution of
\begin{equation}\label{eq:GS}
\Delta_{\mathbb{R}^2}Q-Q+Q^3=0.
\end{equation}
Therefore, for $N < \infty$, we have
\begin{theorem}[Sharp Gagliardo-Nirenberg inequality for the vector function]\label{th3.4v99}
$\forall\, \vec{u} \in \left(  H^1_x \right)^N   $, we have
\begin{align*}
\mathcal{N}(\vec{u}) \le \frac{2(2N-1)}N {\|Q\|_{L^2}^{- 2}}  \|\nabla_x \vec{u} \|_{L_x^2l^2}^2 \|\vec{u}\|_{L_x^2 l^2}^2,
\end{align*}
the equality holds at the positive radial function $\vec{Q} = \Big( \overbrace{\sqrt{\frac1{2N-1}} Q, \cdots, \sqrt{\frac1{2N-1}} Q }^{N-\text{ copy } } \Big)
 \in \left( H^1 \right)^N    \cap \left( C^\infty \right)^N $, where $Q$ is the ground state of the elliptic equation
\begin{align*}
\Delta_{\mathbb{R}^2 } Q - Q = - Q^3.
\end{align*}

\end{theorem}

We now turn to the infinite case when $J = \mathbb{Z}$.
In this case, we can construct a maximizing sequences that converges to $C_\infty$ from the case when $N$ is finite using simple compactness argument.
To see this, we first fix $0<\epsilon\ll 1$ and choose $\vec{\tilde{Q}}= \left\{\tilde{Q}_j \right\}_{j\in\mathbb{Z}}$ such that $C_\infty -  W_\infty \left(\vec{\tilde{Q}} \right)<\frac{\epsilon}{2}$.
In particular, from homogeneity and scaling invariance of $W$, we may assume that
	\[\int_{\mathbb{R}^2}\sum_{j\in\mathbb{Z}} \left|\tilde{Q}_j \right|^2=\int_{\mathbb{R}^2}\sum_{j\in\mathbb{Z}} \left|\nabla \tilde{Q}_j \right|^2=1.\]
Also, we can choose $N=N(\epsilon)>0$ so that
	\[ 2\int_{\mathbb{R}^2}\bigg(\sum_{j=-N}^N  \left|\tilde{Q}_j \right|^2\bigg)^2- \int_{\mathbb{R}^2}\sum_{j=-N}^N \left|\tilde{Q}_j \right|^4 >
 \bigg\{ 2\int_{\mathbb{R}^2}\bigg(\sum_{j\in\mathbb{Z}}  \left|\tilde{Q}_j \right|^2\bigg)^2- \int_{\mathbb{R}^2}\sum_{j\in\mathbb{Z}}  \left|\tilde{Q}_j \right|^4 \bigg\} -\frac{\epsilon}{2} \]
Let $\vec{\tilde{Q}}_N  =  \left\{\tilde{Q}_{j,N} \right\}_{|j|\leq N}$ where $\tilde{Q}_{j,n} :=  \tilde{Q}_j$ for $|j|\leq N$.
Then
	\[ C_\infty-\epsilon< W_\infty \left(\vec{\tilde{Q}} \right)<  W_{2N+1} \left(\vec{\tilde{Q}}_N \right) \leq C_{2N+1}\leq C_\infty.\]
So we can conclude that $C_{2N+1} \nearrow C_ \infty <\infty$ as $N\to\infty$. (Indeed, one can easily see that $C_{N} \nearrow C_ \infty$ as $N\to\infty$.
Finally, from \eqref{eq:QoptimalC}, we have
	\[C_\infty =\frac{4}{\|Q\|_{L^2}^2},\]
and the sharp Gagliardo-Nireberg interpolation inequality (without equality sign) can be written precisely as
	\[2\int_{\mathbb{R}^2}\bigg(\sum_{j\in  \mathbb{Z} } |u_j|^2\bigg)^2- \int_{\mathbb{R}^2}\sum_{j\in  \mathbb{Z} }|u_j|^4 < \frac{4}{\|Q\|_{L^2}^2} \bigg(\int_{\mathbb{R}^2}\sum_{j\in  \mathbb{Z} } |u_j|^2 \bigg)\bigg(\int_{\mathbb{R}^2} \sum_{j\in  \mathbb{Z} } |\nabla u_j|^2 \bigg).\]

Therefore, we obtain
\begin{theorem}[Sharp Gagliardo-Nirenberg inequality for the infinite vector function]\label{th3.8v32}
$\forall\, \vec{u} \in H^1_x l^2$, we have
\begin{align*}
\mathcal{N}(\vec{u}) \le \frac4{\|Q\|_{L^2}^2} \|\nabla_x \vec{u} \|_{L_x^2l^2}^2 \|\vec{u}\|_{L_x^2 l^2}^2,
\end{align*}
the constant is optimal in the sense: $\exists \, \vec{u}^k \in H_x^1 l^2$, s.t.
\begin{align*}
\frac{ \mathcal{N}(\vec{u}^k) }{ \|\nabla_x \vec{u}^k \|_{L_x^2 l^2}^2 \|\vec{u}^k \|_{L_x^2 l^2}} \to \frac4{\|Q\|_{L_x^2}^2}, \text{ as } k \to \infty.
\end{align*}
\end{theorem}

As a direct consequence of Theorem \ref{th3.8v32}, whenever $\left\|\vec{\varphi} \right\|_{L_x^2 l^2}^2 < \frac12 \left\|Q\right\|_{L_x^2}^2$, we have
\begin{align}\label{eq3.0v96}
E(\vec{\varphi}) \ge \frac12 \left( 1 - \frac{\left\|\vec{\varphi} \right\|_{L_x^2 l^2}^2}{\frac12 \left\|Q\right\|_{L_x^2}^2}  \right) \left\|\nabla_x \vec{\varphi} \right\|_{L_x^2 l^2}^2
\gtrsim \left\|\nabla_x \vec{\varphi}\right\|_{L_x^2 l^2}^2.
\end{align}
On the other hand, by Theorem \ref{th3.8v32}, we have
\begin{align*}
E(\vec{\varphi}) \lesssim \left\|\nabla_x \vec{\varphi} \right\|_{L_x^2 l^2}^2 + \left\|\vec{\varphi} \right\|_{L_x^2 l^2}^2 \left\|\nabla_x \vec{\varphi} \right\|_{L_x^2 l^2}^2.
\end{align*}

\section{Reduction to the almost periodic solution}\label{se4v30}

In this section, we show the non-scattering is equivalent to the existence of almost periodic solution. This reduction is standard, we refer to \cite{Ch,Killip-Visan1,TVZ1,TVZ2} and also \cite{CGZ,YZ} for the argument of the resonant nonlinear Schr\"odinger system in the defocusing case. 

By Theorem \ref{th2.4v103}, to prove \eqref{eq1.1} is globally well-poesed and scatters in $L_x^2 h^1$ for $\vec{u}_0 \in L_x^2 h^1$ satisfying $\left\|\vec{u}_0\right\|_{L_x^2 l^2}^2 < \frac12 \|Q\|_{L_x^2}^2$, it suffices to prove that if $\vec{u}$ is a solution of \eqref{eq1.1}, then
\begin{align*}
\|\vec{u} \|_{L_{t,x}^4 l^2} < \infty.
\end{align*}
Define
\begin{align*}
\Lambda(m) = \sup\left\{ \left\|\vec{u} \right\|_{L_{t,x}^4 l^2( I \times \mathbb{R}^2 \times \mathbb{Z})} : \left\|\vec{u}(0) \right\|_{L_x^2 l^2(\RR^2 \times \ZZ)} \le m, \text{ and } \vec{u}(0) \in L_x^2 h^1(\RR^2 \times \ZZ)
\right\}
\end{align*}
where $I$ is the maximal lifespan interval. Let
\begin{align*}
m_0 = \sup \left\{ m: \Lambda(m') < \infty, \forall\, m' < m\right\}.
\end{align*}
If we can prove $m_0 = \frac1{\sqrt2} \|Q\|_{L_x^2}$, then the global well-posedness and scattering are established.
Assume $m_0 < \frac1{\sqrt2} \|Q\|_{L_x^2}$, by following the standard concentration-compactness/rigidity arguments in \cite{TVZ1,TVZ2,Killip-Visan1}, we obtain
\begin{theorem}[Existence of the minimal blowup solution]\label{th4.9v51}
Suppose $m_0 < \frac1{\sqrt2} \|Q\|_{L_x^2} $, there exists a solution
 $\vec{u}\in C_t^0 L_x^2 h^1(I \times \mathbb{R}^2 \times \mathbb{Z}) \cap L_{t,x}^4 l^2 (I \times \mathbb{R}^2 \times \mathbb{Z})$
 of \eqref{eq1.1} with $ \left\| \vec{u} \right\|_{L_x^2 l^2}= m_0$, which is almost periodic in the sense that there exists $(x(t), \xi(t), N(t)) \in \mathbb{R}^2 \times \mathbb{R}^2  \times \mathbb{R}^+$ such that for any $\eta > 0$, there exists $C(\eta) > 0$ satisfying for any $t\in I$,
\begin{align}\label{eq2.10v20}
\int_{|x-x(t)|\ge \frac{C(\eta)}{N(t)}} \left\|\vec{u}(t,x)\right\|_{h^1}^2 \,\mathrm{d}x + \int_{|\xi- \xi(t)|\ge C(\eta) N(t)} \left\|\hat{\vec{u}}(t,\xi)\right\|_{h^1}^2 \,\mathrm{d}\xi < \eta.
 \end{align}
Here $I$ is the maximal lifespan interval. Moreover, we can take $N(0) = 1$, $x(0) = \xi(0) = 0$, $N(t) \le 1$ on $I$, and
\begin{align}\label{eq2.17v20}
|N'(t) | +  |\xi'(t)| \lesssim N(t)^3.
\end{align}
\end{theorem}
By using the argument in \cite{D3,D1,D2,Killip-Visan1}, we have the following results of the almost periodic solution in the above theorem:
\begin{lemma}\label{le2.3v20}
(1) There exists $\delta(\vec{u}) > 0$ such that for any $t_0 \in I$,
\begin{align*}
\left\|\vec{u}\right\|_{L_{t,x}^4 l^2 \left( \left[t_0, t_0 + \frac{\delta}{N(t_0)^2} \right] \times \mathbb{R}^2  \times \ZZ \right)} \sim \left\|\vec{u}\right\|_{L_{t,x}^4 l^2 \left( \left[t_0 - \frac\delta{N(t_0)^2}, t_0 \right] \times \mathbb{R}^2  \times \ZZ \right)} \sim 1.
\end{align*}
(2) If $J$ is a characteristic interval which is defined to be an interval satisfying $\left\|\vec{u}\right\|_{L_{t,x}^4 l^2 (J\times \mathbb{R}^2 \times \ZZ)} = 1$, then for any $t_1,t_2 \in J$, we have $N(t_1) \sim_{m_0} N(t_2)$, and $|\xi(t_1) - \xi(t_2)| \lesssim N(J)$, where $N(J) := \sup\limits_{t\in J} N(t).$ In addition,
\begin{align*}
N(J) \sim \int_{J} N(t)^3 \,\mathrm{d}t \sim \inf\limits_{t\in J} N(t).
\end{align*}
\end{lemma}
By the Strichartz estimates and \eqref{eq2.10v20}, we have
\begin{lemma}\label{le2.6v20}
On the characteristic interval $J$, the following inequality holds 
\begin{align*}
\left\|P_{|\xi - \xi(t)| \ge R N(t)} \vec{u}\right\|_{L_{t,x}^4 l^2 (J \times \mathbb{R}^2 \times \ZZ)}^4 + \int_J \int_{|x-x(t)|\ge R N(t)^{-1}} \left\|\vec{u}(t,x)\right\|_{l^2}^4 \,\mathrm{d}x \mathrm{d}t \le o_R(1).
\end{align*}
\end{lemma}
By the result in \cite{D3,D1,D2,YZ}, we have
\begin{theorem}\label{th1.12v20}
Suppose $\vec{u}(t,x)$ is the almost periodic solution to \eqref{eq1.1} in Theorem \ref{th4.9v51}, with
$\int_0^\infty N(t)^3 \,\mathrm{d}t = K < \infty$. Then for $0 \le s < 3$,
\begin{align*}
\|\vec{u} \|_{L_t^\infty \dot{H}_x^s l^2 ([0,\infty) \times \mathbb{R}^2\times \ZZ)} \lesssim  K^s.
\end{align*}
\end{theorem}
The errors arising from the the Fourier truncation can be estimated for a variety of potentials. We refer to Theorem 7.1 in \cite{YZ} and also Theorem 5.3 in \cite{D1} for a proof of $a(t,x) = \frac{x}{|x|}$, which can be extended to the more general potentials.
\begin{theorem}\label{th1.13v20}
Suppose $\vec{u}$ is the almost periodic solution of \eqref{eq1.1} with $\int_0^T N(t)^3 \,\mathrm{d}t = K$, and there exists a constant $R$ such that
\begin{align*}
| {a}(t,x)| \le R, \ |\nabla_x  {a}(t,x)| \le \frac{R}{|x|}, \ {a}(t,x) = -  {a}(t,-x), \text{ and }
\left\|\partial_t  {a}(t,x) \right\|_{L_x^1(\mathbb{R}^2)} \le R.
\end{align*}
Then the Fourier truncation error arising from $P_{\le CK} \vec{F}(\vec{u}) -\vec{F}(P_{\le CK} \vec{u})$ is bounded by $R o(K)$.
\end{theorem}

\section{Exclusion of the almost periodic solution}\label{se5v30}
In this section, we will exclude the almost periodic solution in Theorem \ref{th4.9v51}.
We consider the cases (i) $\int_0^\infty N(t)^3 \,\mathrm{d}t$ is finite or (ii) infinite separately.
If it is infinite, we exclude the almost periodic solution by exploiting the interaction Morawetz estimate,
otherwise, we exclude the almost periodic solution by exploiting the conservation of energy.

\subsection{Exclusion of the almost periodic solution when $\int_0^\infty N(t)^3 \,\mathrm{d}t = \infty$}
In this subsection, we exclude the case when $\int_0^\infty N(t)^3 \,\mathrm{d}t = \infty$ by the frequency localized interaction Morawetz estimate.
First, we can replace the frequency scale function $N(t)$ by a slowly varying frequency scale function of the almost periodic solution in Theorem \ref{th4.9v51}. Following the argument in \cite{D4}, we can use a smoothing algorithm developed by B. Dodson, and replace $N(t)$ with a slowly varying $\tilde{N}(t)$ and
 $\tilde{N}(t) \le N(t)$. Furthermore, by the construction, we can make sure
\begin{align}\label{eq5.1v100new}
\frac{\left| \tilde{N}'(t)\right|}{ \tilde{N}(t)^3} \lesssim 1, \forall\, t > 0 ,
\end{align}
 and if $ \tilde{N}'(t)\ne 0$, then $ \tilde{N}(t) = N(t)$.

By applying the argument in \cite{D4}, we get
\begin{lemma}\label{le6.3v20}
For any $\delta > 0$, we can take a smoother $\tilde{N}(t)$ such that
\begin{align}\label{eq5.1v100}
\liminf\limits_{T\to \infty} \frac{ \int_0^T | \tilde{N} '(t)| \,\mathrm{d}t}
{ \int_0^T  \tilde{N} (t) \mathcal{N} \left( \overrightarrow{P_{\le CK} u} \right)(t) \,\mathrm{d}t } \le \delta .
\end{align}
\end{lemma}
In the following, we still take $N(t)$ as $\tilde{N}(t)$. And therefore, the $N(t)$ satisfies \eqref{eq5.1v100new} and \eqref{eq5.1v100}, and we will exclude the case when $\int_0^\infty N(t)^3 \,\mathrm{d}t = \infty$ by the interaction Morawetz estimate. The interaction Morawetz estimate is developed by
J. Colliander, M. Keel, G. Staffilani, H. Takaoka, and T. Tao \cite{CKSTT1}, 
 which is used to prove the scattering of the nonlinear Schr\"odinger equation \cite{CKSTT0,TVZ1,D1,D2,D3} in the non-radial case. By direct calculation, we obtain the following lemma on the interaction Morawetz estimate for \eqref{eq1.1}.
\begin{lemma}\label{MoraEst}
For a weight function $a: \mathbb{R}^2 \to \mathbb{R} $,
let
\begin{align*}
M(t) = 2  \sum\limits_{j,j'\in \mathbb{Z}} \int_{\mathbb{R}^2} \int_{\mathbb{R}^2} |u_{j'}(t,y)|^2 \nabla_x a(x-y) \cdot \Im( \bar{u}_j \nabla u_j)(t,x) \,\mathrm{d}x \mathrm{d}y,
\end{align*}
we have
\begin{align*}
M'(t) & = \sum\limits_{j,j' \in \mathbb{Z}} \sum\limits_{1 \le k, l \le 2}
4 \int_{\mathbb{R}^2} \int_{\mathbb{R}^2} \partial_{x_k} \partial_{x_l} a(x-y) \Re ( \partial_{x_k} u_j \partial_{x_l} \bar{u}_j)(t,x) |u_{j'}(t,y)|^2 \,\mathrm{d}x \mathrm{d}y
\\
& \ - \sum\limits_{j,j'\in \mathbb{Z}} \int_{\mathbb{R}^2} \int_{\mathbb{R}^2}  \Delta_x^2 a(x-y) |u_j(t,x)|^2 |u_{j'}(t,y)|^2 \,\mathrm{d}x \mathrm{d}y
\\
& \ - \sum\limits_{j,j'\in \mathbb{Z}} \int_{\mathbb{R}^2} \int_{\mathbb{R}^2} \Delta_x a(x-y) |u_{j'}(t,y)|^2 \sum\limits_{\mathcal{R}(j)} ( \bar{u}_j u_{j_1} \bar{u}_{j_2} u_{j_3} )(t,x) \,\mathrm{d}x \mathrm{d}y
\\
& \ - 4 \sum\limits_{j,j' \in \mathbb{Z}}  \sum\limits_{1 \le k, l \le 2}
\int_{\mathbb{R}^2} \int_{\mathbb{R}^2} \partial_{x_l} a(x-y) \Im ( \bar{u}_j \partial_{x_l} u_j)(t,x) \partial_{y_k} \Im ( \bar{u}_{j'} \partial_{y_k} u_{j'})(t,y) \,\mathrm{d}x \mathrm{d}y.
\end{align*}
\end{lemma}
For any $T > 0$, define
\begin{align*}
K(T) = \int_0^T N(t)^3 \,\mathrm{d}t.
\end{align*}

\begin{theorem}\label{th5.1v20}
If $\vec{u}$ is the almost periodic solution to \eqref{eq1.1} satisfying $\int_0^\infty N(t)^3 \,\mathrm{d}t = \infty$
in Theorem \ref{th4.9v51}, then $\vec{u} =0$.
\end{theorem}
\begin{proof}

Let $\varphi$ be a $C_0^\infty$ radial function with
\begin{align*}
\varphi(x) =
\begin{cases}
1,\ |x|\le R- \sqrt{R},\\
0, \ |x|\ge R.
\end{cases}
\end{align*}
Let
\begin{align*}
\phi(x) = \frac1{2\pi R^2} \int_{\mathbb{R}^2} \varphi( |x- s|) \varphi(|s|) \,\mathrm{d}s,
\end{align*}
and define
\begin{align*}
\psi_{ R N(t)^{- 1 } } (r) = \frac1r \int_0^r \phi \left(\frac{ N(t)  s}R\right) \,\mathrm{d}s.
\end{align*}
We have
\begin{align}\label{eq5.1v80}
 \left|\psi_{RN(t)^{- 1}}'(r)  \right| \lesssim \min \left( \left(RN(t)^{- 1 } \right)^{- \frac32},  \left(R N(t)^{- 1 } \right)^{- \frac52} r, r^{-1}  \right).
\end{align}

Define the frequency localized interaction Morawetz action
\begin{align*}
M(t) = \sum\limits_{j,j' \in \mathbb{Z}} \iint_{\mathbb{R}^2 \times \mathbb{R}^2}  \psi_{ R N(t)^{- 1 }} \left(|x-y|\right) N(t) (x-y)  \cdot \Im \left( \overline{ I  u_j(t,x)} \nabla_x I u_j(t,x) \right) \left| I u_{j'} (t,y)\right|^2 \,\mathrm{d}x \mathrm{d}y,
\end{align*}
where $I= P_{\le CK}$.

Notice that $\psi_{ R N(t)^{- 1 }} \left(|x-y| \right) N(t) (x-y) $ satisfies the conditions of Theorem \ref{th1.13v20} with $R$ replaced by $R^2$.  By Lemma \ref{MoraEst} and Theorem \ref{th1.13v20}, we obtain
\begin{align}
& M'(t) = \notag \\
 & - 4 \sum\limits_{j,j'\in \mathbb{Z}} \iint_{\mathbb{R}^2 \times\mathbb{R}^2}  \psi_{ R N(t)^{- 1 }} (|x-y|) N(t) (x-y) \cdot \nabla_x \Re\left( \overline{ \nabla_x Iu_j(t,x)} \otimes \nabla_x Iu_j(t,x) \right) \left|Iu_{j'}(t,y) \right|^2 \,\mathrm{d}x \mathrm{d}y \label{eq5.5v20}\\
 & - 4 \sum\limits_{j,j'\in \mathbb{Z}} \iint_{\mathbb{R}^2 \times \mathbb{R}^2}  \psi_{R N(t)^{-  1 }} (|x-y|) N(t) (x-y) \cdot \Im \left( \overline{Iu_j(t,x)} \nabla_x Iu_j(t,x) \right) \nabla_y \Im \left( \overline{Iu_{j'}(t,y)} \nabla_y Iu_{j'}(t,y)  \right) \,\mathrm{d}x \mathrm{d}y \label{eq5.6v20}\\
 & + \sum\limits_{j,j'\in \mathbb{Z}} \iint_{\mathbb{R}^2 \times \mathbb{R}^2}  \psi_{R N(t)^{-  1 }}  ( |x-y|) N(t) (x-y) \cdot  \nabla_x \Delta_x \left( |Iu_j(t,x)|^2\right) \left|Iu_{j'}(t,y) \right|^2 \,\mathrm{d}x \mathrm{d}y \label{eq5.7v20}\\
 & +  \sum\limits_{j,j'\in \mathbb{Z}} \iint_{\mathbb{R}^2 \times \mathbb{R}^2}  \psi_{ R N(t)^{-  1 }} ( |x-y|) N(t) (x-y) \cdot \nabla_x
 \left( \overline{Iu_j} \vec{F}_j (\overrightarrow{Iu})  \right)  (t,x)  \left|Iu_{j'}(t,y) \right|^2 \,\mathrm{d}x \mathrm{d}y \label{eq5.8v20}\\
 & +  \sum\limits_{j,j'\in \mathbb{Z}} \iint_{\mathbb{R}^2 \times \mathbb{R}^2}  \frac{d}{dt} \left( \psi_{RN(t)^{-  1 } } (|x-y|)N(t)(x-y)\right) \cdot \Im \left( \overline{Iu_j(t,x)} \nabla_x Iu_j(t,x)\right) \left|Iu_{j'} (t,y) \right|^2 \mathrm{d}x \mathrm{d}y  \\
&  + \mathcal{E}(t), \notag
\end{align}
with
\begin{align}\label{eq5.7v89}
\int_0^T \mathcal{E}(t) \,\mathrm{d}t \lesssim R^2 o(K).
\end{align}
In \eqref{eq5.5v20}, $\otimes$ represents
$\sum\limits_{k,k'\in \mathbb{Z}} (x_{k'} - y_{k'}) \partial_k \left( \overline{\partial_k I u_j } \partial_{k'} I u_j \right)$.
Integrating by parts, we have
\begin{align}\label{eq5.9v20}
\eqref{eq5.8v20} &
 =  -  2 \sum\limits_{j,j'\in \mathbb{Z}} \iint_{\mathbb{R}^2 \times \mathbb{R}^2} \psi_{ R N(t)^{- 1 }} (|x-y|)
 N(t) \sum\limits_{\mathcal{R}(j)} \left( \overline{Iu_j} I u_{j_1} \overline{I u_{j_2} } I u_{j_3} \right)(t,x) |Iu_{j'} (t,y)|^2 \,\mathrm{d}x \mathrm{d}y \\
& \ -   \sum\limits_{j,j'\in \mathbb{Z}} \iint_{\mathbb{R}^2 \times \mathbb{R}^2}  \psi_{R N(t)^{- 1 }}'( |x-y|) N(t) |x-y| \sum\limits_{\mathcal{R}(j)} \left( \overline{Iu_j} I u_{j_1} \overline{I u_{j_2} } I u_{j_3} \right)(t,x)  \left|Iu_{j'} (t,y) \right|^2 \,\mathrm{d}x \mathrm{d}y, \notag
\end{align}
and
\begin{align}\label{eq5.10v20}
\eqref{eq5.7v20} = - \sum\limits_{j,j'\in \mathbb{Z}} \iint_{\mathbb{R}^2 \times \mathbb{R}^2}  \Delta \left(   \psi_{ R N(t)^{- 1 }} \left( |x-y|\right) + \phi\left( \frac{N(t)|x-y|}R\right) \right) \left|Iu_j(t,x) \right|^2  \left|Iu_{j'} (t,y) \right|^2 \,\mathrm{d}x \mathrm{d}y.
\end{align}
The gradient vector can be decomposed into a radial component and an angular component.
Now let $\nabla_{r,y}$ be the radial derivative with respect to $y$, i.e. $\nabla_{r,y}  = \frac{x-y }{ | x-y|} \nabla_{x} $ and $\not\nabla_y$ be the angular component of $\nabla$. Then by integrating by parts, we have
\begin{align}
& \eqref{eq5.5v20}+ \eqref{eq5.6v20}  \notag\\
= & \quad  4 \sum\limits_{j,j'\in \mathbb{Z}} \iint_{\mathbb{R}^2 \times \mathbb{R}^2}  \phi\left( \frac{|x-y| N(t) }R\right)
N(t)  \left|\nabla Iu_j(t,x) \right|^2  \left|Iu_{j'} (t,y) \right|^2 \,\mathrm{d}x \mathrm{d}y  \notag \\
& \quad - 4 \sum\limits_{j,j' \in \mathbb{Z}} \iint_{\mathbb{R}^2 \times \mathbb{R}^2}  \phi \left( \frac{ |x-y| N(t) }R\right) N(t)
\Im \left( \overline{Iu_j(t,x)} \nabla_x Iu_j(t,x) \right) \cdot \Im \left( \overline{Iu_{j'} (t,y)} \nabla_y Iu_{j'} (t,y)\right)
 \,\mathrm{d}x \mathrm{d}y \label{eq5.11v20}\\
& \quad + 4 \sum\limits_{j,j'\in \mathbb{Z}} \iint_{\mathbb{R}^2\times \mathbb{R}^2}  \left( \psi_{ R N(t)^{-  1 }}
\left(|x-y|\right) - \phi\left( \frac{|x-y| N(t) }R\right)\right)  N(t)  \left|\not\nabla_y Iu_j\left(t,x\right) \right|^2  \left|Iu_{j'} (t,y) \right|^2
 \,\mathrm{d}x \mathrm{d}y \notag  \\
& \quad - 4 \sum\limits_{j,j'\in \mathbb{Z}} \iint_{\mathbb{R}^2 \times \mathbb{R}^2}  \left( \psi_{ R N(t)^{-  1 }} \left(|x-y|\right) - \phi\left( \frac{|x-y| N(t) }R\right)\right) N(t)  \Im\left( \overline{Iu_j(t,x)} {\not\nabla}_y Iu_j(t,x)\right) \notag \\
& \qquad  \qquad \cdot \Im \left( \overline{Iu_{j'} (t,y)} {\not\nabla}_x  Iu_{j'} (t,y) \right) \,\mathrm{d}x \mathrm{d}y. \label{eq5.12v20}
\end{align}
Because $\psi_R$ and $\phi$ are radial functions, we have $\eqref{eq5.12v20} \ge 0$.

Because \eqref{eq5.11v20} is Galilean invariant, we can take 
 a Galilean transform to eliminate 
  the second term in \eqref{eq5.11v20}.
For any $s \in \mathbb{R}^2$, taking 
 $\xi(s)\in \mathbb{R}^2$ such that
\begin{align}\label{eq5.14v20}
\int_{\mathbb{R}^2}  \varphi\left( \left|\frac{N(t) x}R - s\right|\right) \Im \left( e^{ix \xi(s)} \overline{Iu_j} \nabla \left(e^{-i x \xi(s)} Iu_j \right)\right)(t,x) \,\mathrm{d}x = 0.
\end{align}
Then,
we get
\begin{align}
 & \int_0^T
M'(t)  \,\mathrm{d}t \notag \\
  \ge & \ 4 \int_0^T \sum\limits_{j,j' \in \mathbb{Z}} \iint \frac1{ 2\pi  R^2} \int \varphi \left( \left|\frac{x N(t)}R - s \right| \right)
  \varphi\left( \left|\frac{y N(t)}R - s \right| \right)
 N(t) \notag \\
 & \qquad \cdot  \left| \nabla \left(e^{-i x \xi(s)} Iu_j(t,x) \right)\right|^2 \left|Iu_{j'}(t,y) \right|^2 \,\mathrm{d}x \mathrm{d}y \mathrm{d}t \mathrm{d}s
  \label{eq5.12v86}  \\
& -  \int_0^T \sum\limits_{j,j' \in \mathbb{Z}} \iint \Delta \left( \psi_{R N(t)^{-  1 }} \left( |x- y| \right) + \phi\left( \frac{ |x - y|N(t)}R \right) \right) N(t)   \left|Iu_j(t,x) \right|^2  \left|Iu_{j'}(t,y) \right|^2 \,\mathrm{d}x \mathrm{d}y  \mathrm{d}t
\label{eq5.13v86}   \\
& - 2 \int_0^T
 \sum\limits_{j, j' \in \mathbb{Z}} \iint \psi_{RN(t)^{- 1 }} (|x- y|) N(t) \sum\limits_{\mathcal{R}(j)} \left( \overline{Iu_j} Iu_{j_1}
 \overline{Iu_{j_2}} Iu_{j_3} \right)(t,x)  \left|Iu_{j'} (t,y) \right|^2 \,\mathrm{d}x \mathrm{d}y \mathrm{d}t
 \label{eq5.14v86} \\
& - \int_0^T
 \sum\limits_{j,j' \in \mathbb{Z} } \iint \psi_{RN(t)^{-  1 }}' (|x- y|) N(t) |x- y| \sum\limits_{\mathcal{R}(j)} \left( \overline{Iu_j} Iu_{j_1} \overline{Iu_{j_2}} Iu_{j_3} \right)(t,x)  \left|Iu_{j'}(t,y) \right|^2 \,\mathrm{d}x \mathrm{d}y \mathrm{d}t
 \label{eq5.15v86}  \\
& +  \int_0^T \sum\limits_{j,j' \in \mathbb{Z}} \iint \frac{d}{dt} \left( \psi_{RN(t)^{-  1 }} (|x- y|) N(t)(x- y) \right) \Im \left( \overline{Iu_j(t,x) } \nabla_x Iu_j(t,x) \right)  \left|Iu_{j'}(t,y) \right|^2
\,\mathrm{d}x \mathrm{d}y \mathrm{d}t \label{eq5.17v89}\\
& + \int_0^T \mathcal{E}(t) \,\mathrm{d}t .  \label{eq5.18v89}
\end{align}
We will first consider
 the estimate of \eqref{eq5.13v86}. We note for $r = |x|$, we have
\begin{align*}
\psi_{RN(t)^{- 1 }}' (r) = \frac{ \phi\left( \frac{N(t)r}R  \right) - \psi_{RN(t)^{-  1 }} (r)}r,
\end{align*}
and
\begin{align*}
\psi_{RN(t)^{- 1 }}''(r) 
= \frac{ \frac{N(t)}R \phi' \left( \frac{N(t)r}R \right) - 2 \psi_{RN(t)^{- 1 }}' (r) }r.
\end{align*}
Therefore, we have
\begin{align}
& \Delta \left( \psi_{R N(t)^{- 1 } } (|x|) + \phi \left( \frac{ N(t) |x|}R \right)  \right) \notag \\
= &  \psi_{RN(t)^{- 1 }}''(r) + \frac{N(t)^2}{R^2} \phi'' \left( \frac{N(t) r}R \right) +
\frac1r \psi_{RN(t)^{-  1 }}'(r) + \frac1r \frac{N(t)}R \phi'  \left( \frac{N(t) r}R \right)\notag \\
= & \frac1r  \left( \frac{N(t) }R \phi' \left( \frac{ N(t)r} R \right) - 2 \psi_{RN(t)^{- 1 }}' (r) \right) + \frac{N(t)^2}{R^2}
\phi'' \left( \frac{N(t)r }R \right) + \frac1r \psi_{RN(t)^{- 1 }}' ( r) + \frac1r \frac{N(t)}R \phi'  \left( \frac{ N(t)r}R \right). \label{eq5.18v88}
\end{align}
By
\begin{align*}
 \left| \phi' \left( \frac{N(t)r }R  \right) \right | \lesssim \min \left( R^{- \frac32} N(t)^\frac32, R^{- \frac52} N(t)^\frac52 r \right),
\end{align*}
we see
\begin{align}\label{eq5.19v88}
\frac{N(t)}{rR}  \left| \phi' \left( \frac{N(t) r} R \right)  \right| \lesssim \frac{N(t) }{rR} R^{- \frac52} N(t)^\frac52 r \lesssim N(t)^\frac72 R^{- \frac72}.
\end{align}
By $ \left| \phi'' \left( \frac{N(t) r} R \right) \right| \lesssim R^{- \frac52} N(t)^\frac52$, we have
\begin{align} \label{eq5.20v88}
\left| \frac{N(t)^2 }{R^2} \phi'' \left( \frac{N(t)r}R \right) \right | \lesssim \frac{N(t)^2 }{R^2} R^{- \frac52} N(t)^\frac52 \sim R^{- \frac92} N(t)^\frac92.
\end{align}
By
\begin{align}\label{eq5.19v89}
\left|\psi_{RN(t)^{- 1  }}' (r) \right|  \lesssim \min \left( R^{- \frac32} N(t)^\frac32, R^{- \frac52} N(t)^\frac52 r, r^{-1 } \right),
\end{align}
we have
\begin{align}\label{eq5.21v88}
\left| \frac1r \psi_{RN(t)^{-  1  }}' (r) \right| \lesssim \frac1r R^{- \frac52 } N(t)^\frac52 r \sim R^{- \frac52 }N(t)^\frac52.
\end{align}

Therefore, by \eqref{eq5.18v88}, \eqref{eq5.19v88}, \eqref{eq5.20v88}, and \eqref{eq5.21v88}, we have for $R$ large enough,
\begin{align}\label{eq5.22v88}
\left| \Delta \left( \psi_{RN(t)^{- 1 } } (r) + \phi \left( \frac{N(t) r}R \right) \right) \right| \lesssim N(t)^\frac72 R^{- \frac72 } + R^{- \frac92} N(t)^\frac92 + R^{- \frac52 } N(t)^\frac52 \lesssim R^{- \frac52} N(t)^\frac52.
\end{align}
Then, by \eqref{eq5.22v88} and $N(t) \le 1$, we have
\begin{align}
 & \int_0^T   \bigg| - \sum\limits_{j,j' \in \mathbb{Z}} \iint \Delta \left( \psi_{RN(t)^{- 1 } } ( |x - y| ) + \phi \left( \frac{N(t) | x- y| } R \right) \right)  N(t) \left|Iu_j(t,x) \right|^2 \left|Iu_{j'}(t,y) \right|^2 \,\mathrm{d}x \mathrm{d}y \bigg| \mathrm{d}t \notag  \\
\lesssim & \int_0^T R^{- \frac52 } N(t)^\frac52 \int \sum\limits_j \left|Iu_j(t,x)\right|^2 \,\mathrm{d}x N(t) \int  \sum\limits_{j'}   \left|Iu_{j'}(t,y) \right|^2 \,\mathrm{d}y \mathrm{d}t \notag \\
\lesssim & R^{- \frac52} \left\|\vec{u} \right\|_{L_x^2 l^2}^4\int_0^T N(t)^\frac72 \,\mathrm{d}t
\lesssim R^{- \frac52} \left\|\vec{u} \right\|_{L_x^2 l^2}^4 \int_0^T N(t)^3 \,\mathrm{d}t
\sim R^{- \frac52} \left\|\vec{u} \right\|_{L_x^2 l^2}^4 K. \label{eq5.25v89new}
\end{align}
%

We now turn to the estimate of 
 \eqref{eq5.15v86}. On the characteristic interval $J_k$ of $[0,T]$, we have
\begin{align}
 & \int_{J_k} \bigg|  \iint \psi_{R N(t)^{- 1 }}' \left( |x- y| \right) N(t)  |x- y|
  \sum\limits_{j,j'} \sum\limits_{\mathcal{R}(j)}  \left( \overline{Iu_j} Iu_{j_1} \overline{Iu_{j_2}} Iu_{j_3}  \right)(t,x) \left|Iu_{j'}(t,y) \right|^2
  \,\mathrm{d}x \mathrm{d}y \bigg| \,\mathrm{d}t  \notag \\
\lesssim & \int_{J_k} \bigg( \iint_{|x- x(t)| \ge \frac{ R}{N(t)} } + \iint_{|y - x(t)|\ge \frac{R}{N(t)} } + \iint_{|x- x(t)|\le  \frac{R}{N(t)}
, \atop
|y - x(t)|\le \frac{R}{N(t)} }  \bigg)  \left|\psi_{ R N(t)^{- 1 }}' \left(|x- y| \right)  \right| N(t) |x- y| \notag  \\
& \quad \cdot \sum\limits_{j,j'} \sum\limits_{\mathcal{R}(j)}
 \left(\overline{Iu_j} Iu_{j_1} \overline{Iu_{j_2}} Iu_{j_3}  \right)(t,x)
  \left|Iu_{j'}(t,y) \right|^2 \,\mathrm{d}x \mathrm{d}y \mathrm{d}t. \label{eq5.22v89new}
\end{align}
On the integral domain $ \left\{|x - x(t)|\le \frac{R}{N(t)} , |y - x(t)| \le \frac{R}{N(t)}  \right\}$, by \eqref{eq5.19v89},
we have 
\begin{align*}
& \int_{J_k} \iint_{ |x - x(t) | \le \frac{R}{N(t)} \atop |y - x(t)| \le \frac{R}{N(t)}}
\sum\limits_{j,j' \in \mathbb{Z}}
\left|\psi_{RN(t)^{- 1 }}'  \left(|x- y| \right) \right| N(t) |x - y|   \sum\limits_{\mathcal{R}(j)}
  \left( \overline{Iu_j} Iu_{j_1} \overline{Iu_{j_2}} I u_{j_3}  \right)(t,x)  \left|Iu_{j'}(t,y) \right|^2 \,\mathrm{d}x \mathrm{d}y \mathrm{d}t
\\
\lesssim  &  \int_{J_k} \iint_{| x- x(t) | \le \frac{R}{N(t)} \atop  |y - x(t)|
 \le \frac{R}{N(t)} } N(t)^\frac52  R^{- \frac32 } |x - y| \sum\limits_{j,j' \in \mathbb{Z}}
 \sum\limits_{\mathcal{R}(j)} \left( \overline{Iu_j} Iu_{j_1} \overline{Iu_{j_2}} Iu_{j_3} \right)(t,x) |Iu_{j'}(t,y)|^2 \,\mathrm{d}x \mathrm{d}y \mathrm{d}t \\
\lesssim & R^{- \frac32} R \int_{J_k} \int \frac{N(t)^\frac52  }{N(t)} \sum\limits_{j,j' \in \mathbb{Z}} \sum\limits_{\mathcal{R}(j)} \left( \overline{Iu_j} Iu_{j_1} \overline{Iu_{j_2}} Iu_{j_3} \right)(t,x)  \left|Iu_{j'}(t,y) \right|^2 \,\mathrm{d}x \mathrm{d}y \mathrm{d}t \\
\lesssim & R^{- \frac12} \int_{J_k}  \sum\limits_{j' \in \mathbb{Z}} \int \sum\limits_{j \in \mathbb{Z}} \sum\limits_{\mathcal{R}(j)} \left( \overline{Iu_j} Iu_{j_1} \overline{Iu_{j_2}} Iu_{j_3} \right)(t,x) N(t)^\frac32  \,\mathrm{d}x \mathrm{d}t \\
\lesssim & R^{- \frac12} \|\vec{u}\|_{L_y^2 l^2}^2 \int_{J_k} \int \sum\limits_{j \in \mathbb{Z}} \sum\limits_{\mathcal{R}(j)} \left( \overline{Iu_j} Iu_{j_1} \overline{Iu_{j_2}} Iu_{j_3} \right)(t,x) N(t)^\frac32
 \,\mathrm{d}x \mathrm{d}t .
\end{align*}
Taking the summation of the characteristic interval $J_k$ in the above estimate, we have by $N(t) \le 1$ and Lemma \ref{le2.3v20},
\begin{align}
& \int_0^T \bigg| - \sum\limits_{j,j' \in \mathbb{Z}} \iint_{ |x - x(t) | \le \frac{ R}{N(t)}  \atop |y - x(t)| \le  \frac{R}{N(t)} }
 \psi_{RN(t)^{- 1 }}' ( |x - y|) N(t) | x - y| \sum\limits_{\mathcal{R}(j)}  \left( \overline{Iu_j} Iu_{j_1} \overline{Iu_{j_2}} Iu_{j_3} \right)(t,x) \left|Iu_{j'}(t,y) \right|^2 \,\mathrm{d}x \mathrm{d}y \bigg| \mathrm{d}t \notag  \\
= & \sum\limits_{J_k \subseteq [0, T]} \int_{J_k} \bigg| - \sum\limits_{j,j' \in \mathbb{Z}} \iint_{|x - x(t) | \le \frac{ R}{N(t)}   \atop  |y - x(t) | \le  \frac{R}{N(t)}  }
\psi_{RN(t)^{- 1 }}' ( | x  - y |) N(t) |x - y| \notag  \\
& \quad \cdot
 \sum\limits_{\mathcal{R}(j)} \left( \overline{Iu_j} Iu_{j_1} \overline{Iu_{j_2}} I u_{j_3} \right)(t,x) \left|Iu_{j'}(t,y) \right|^2 \,\mathrm{d}x \mathrm{d}y \bigg| \mathrm{d}t \notag \\
\lesssim & \sum\limits_{J_k \subseteq [ 0, T]} R^{- \frac12}  \left\|\vec{u}\right\|_{L_y^2 l^2}^2 \int_{J_k} \int \sum\limits_{j \in \mathbb{Z}} \sum\limits_{\mathcal{R}(j)}  \left( \overline{Iu_j} I u_{j_1} \overline{Iu_{j_2}} Iu_{j_3}  \right)(t,x) N(t)^\frac32
 \,\mathrm{d}x \mathrm{d}t \notag  \\
\lesssim & 
R^{- \frac12}  \left\|\vec{u} \right\|_{L_y^2 l^2}^2 \sum\limits_{J_k \subseteq [0,T]} N(J_k)
\lesssim   R^{- \frac12}  \left\|\vec{u} \right\|_{L_y^2 l^2}^2 \sum\limits_{J_k \subseteq [0,T]} \int_{J_k} N(t)^3 \,\mathrm{d}t
\sim R^{- \frac12}  \left\| \vec{u} \right\|_{L_y^2 l^2}^2 K.\label{eq5.22v89}
\end{align}

We now consider
\begin{align*}
 \int_{J_k}     \iint_{|x- x(t)| \ge \frac{  R}{N(t)}}     \left|\psi_{ R N(t)^{- 1 }} '( |x- y|)  \right| N(t)
  |x- y| \cdot \sum\limits_{j,j'} \sum\limits_{\mathcal{R}(j)}
   \left( \overline{Iu_j} Iu_{j_1} \overline{Iu_{j_2}} Iu_{j_3} \right)(t,x)
    \left|Iu_{j'}(t,y) \right|^2 \,\mathrm{d}x \mathrm{d}y \mathrm{d}t.
\end{align*}
in \eqref{eq5.22v89new}.
By the fact $ \left|\psi_{RN(t)^{- 1 }}' (r) \right| \lesssim r^{- 1}$, we have 
\begin{align}\label{eq5.30v102}
 \left|\psi_{RN(t)^{- 1 }}' ( |x - y|)  \right|  N(t) |x - y| \lesssim |x - y|^{-1} N(t) |x - y| \sim N(t),
\end{align}
then by \eqref{eq5.30v102} and Lemma \ref{le2.6v20}, we have
\begin{align*}
&  \int_{J_k} \iint_{|x - x(t) | \ge \frac{R}{N(t)} } \sum\limits_{j,j' \in \mathbb{Z}} \psi_{RN(t)^{- 1 }}' (|x - y|) N(t) | x - y|
\sum\limits_{\mathcal{R}(j)} \left( \overline{Iu_j} Iu_{j_1} \overline{Iu_{j_2}} Iu_{j_3}  \right)(t,x)  \left|Iu_{j'}(t,y) \right|^2 \,\mathrm{d}x \mathrm{d} y \mathrm{d}t \\
\lesssim  & \int_{J_k} \iint_{|x - x(t)| \ge \frac{R}{N(t)} } \sum\limits_{j,j' \in \mathbb{Z}} N(t) \sum\limits_{\mathcal{R}(j)}
\left( \overline{Iu_j} Iu_{j_1} \overline{Iu_{j_2}} Iu_{j_3}  \right)(t,x)   \left|Iu_{j'}(t,y) \right|^2 \,\mathrm{d}x \mathrm{d}y \mathrm{d}t \\
\lesssim & \int_{J_k} \sum\limits_{j' \in \mathbb{Z}} \int   \left|Iu_{j'}(t,y) \right|^2 \,\mathrm{d}y \cdot
 \int_{|x - x(t)| \ge \frac{R}{N(t)} } N(t) \sum\limits_{j \in \mathbb{Z}} \sum\limits_{\mathcal{R}(j)}
  \left( \overline{Iu_j} Iu_{j_1} \overline{Iu_{j_2}} Iu_{j_3}  \right) (t,x) \,\mathrm{d}x \mathrm{d}t \\
\lesssim & \|\vec{u}\|_{L_y^2 l^2}^2 \int_{J_k}
\int_{|x - x(t)| \ge \frac{R}{N(t)}} N(t) \sum\limits_{j \in \mathbb{Z}} \sum\limits_{\mathcal{R}(j)}  \left( \bar{u}_j u_{j_1} \bar{u}_{j_2} u_{j_3} \right)(t,x) \,\mathrm{d}x \mathrm{d}t \\
& +  \left\|\vec{u} \right\|_{L_y^2 l^2}^2  \left( \int_{J_k} \int \sum\limits_{j \in \mathbb{Z}}
\left| P_{\ge CK} u_j(t,x) \right|^4 \,\mathrm{d}x \mathrm{d}t  \right)^\frac14  \left\|\vec{u} \right\|_{L_{t,y}^4 l^4(J_k )}^3 N(J_k 
) \\
\lesssim &  \left\|\vec{u} \right\|_{L_y^2 l^2}^2  \int_{J_k} \int_{|x - x(t) | \ge \frac{R}{N(t)} } N(t) \sum\limits_{j \in \mathbb{Z}} \sum\limits_{\mathcal{R}(j)}  \left( \bar{u}_j u_{j_1} \bar{u}_{j_2} u_{j_3}  \right)(t,x) \,\mathrm{d}x \mathrm{d}t
+  \left\|\vec{u} \right\|_{L_y^2 l^2}^2  \left\|P_{\ge CK} \vec{u} \right\|_{L_{t,x}^4 l^2}  \left\|\vec{u}  \right\|_{L_{t,x}^4 l^2}^3 N(J_k 
) \\
\lesssim &  \left\|\vec{u} \right\|_{L_y^2 l^2}^2 N(J_k 
)  \int_{J_k}  \int_{|x - x(t)| \ge\frac{R}{N(t)}} \sum\limits_{j \in \mathbb{Z}} \sum\limits_{\mathcal{R}(j)}  \left( \bar{u}_j u_{j_1} \bar{u}_{j_2} u_{j_3} \right)(t,x) \,\mathrm{d}x \mathrm{d}t
+  \left\|\vec{u} \right\|_{L_y^2 l^2}^2  \left\|P_{ | \xi - \xi(t) | \ge { C(R) K} N(t) }  \vec{u} \right\|_{L_{t,x}^4 l^2 (J_k)} N(J_k 
) \\
\le &  o_R(1) N(J_k 
).
\end{align*}

Thus, taking summation of the characteristic interval $J_k$ of $[0,T]$, by Lemma \ref{le2.3v20}, we have
\begin{align}
 &  \int_0^T   \bigg|- \sum\limits_{j,j' \in \mathbb{Z}} \iint_{|x - x(t) | \ge \frac{R}{N(t)} } \psi_{RN(t)^{- 1 }}' ( | x- y|) N(t) | x - y| \sum\limits_{\mathcal{R}(j)}  \left( \overline{Iu_j} Iu_{j_1} \overline{Iu_{j_2}} Iu_{j_3} \right)(t,x) \left|Iu_{j'}(t,y) \right|^2 \,\mathrm{d}x \mathrm{d}y  \bigg| \mathrm{d}t \notag  \\
= & \sum\limits_{J_k \subseteq [0,T]} \int_{J_k} \bigg| - \sum\limits_{j,j' \in \mathbb{Z}} \iint \psi_{RN(t)^{-  1 }}' ( | x- y|) N(t) | x- y|
\sum\limits_{\mathcal{R}(j)}  \left( \overline{Iu_j} Iu_{j_1} \overline{Iu_{j_2}} Iu_{j_3} \right)(t,x) \left|Iu_{j'}(t,y) \right|^2 \,\mathrm{d}x \mathrm{d}y \bigg| \mathrm{d}t  \notag \\
\lesssim & o_R (1)\sum\limits_{J_k \subseteq [0, T]}  N(J_k 
)
\sim o_R(1) \sum\limits_{J_k \subseteq [0,T]} \int_{J_k} N(t)^3 \,\mathrm{d}t \sim o_R(1)K. \label{eq5.23v89}
\end{align}

In the meantime, on the characteristic interval $J_k$, by \eqref{eq5.30v102}, \eqref{eq2.10v20}, and Lemma \ref{le2.6v20}, we have
\begin{align*}
& \int_{J_k} \iint_{|y - x(t)|  \ge \frac{R}{N(t)} }  \left|\psi_{ R N(t)^{- 1 }} '( |x - y|)  \right| N(t)  | x- y| \sum\limits_{j,j'} \sum\limits_{\mathcal{R}(j)}  \left( \overline{Iu_j} Iu_{j_1} \overline{Iu_{j_2}} Iu_{j_3}  \right)(t,x) \left|Iu_{j'}(t,y) \right|^2 \,\mathrm{d}x \mathrm{d}y \mathrm{d}t \\
\lesssim & \int_{J_k}  N(t) \iint_{|y - x(t) | \ge \frac{R}{N(t)} } \sum\limits_{j,j'} \left( \overline{Iu_j} Iu_{j_1} \overline{Iu_{j_2}} Iu_{j_3} \right)(t,x) \left|Iu_{j'}(t,y) \right|^2 \,\mathrm{d}x \mathrm{d}y \mathrm{d}t \\
\lesssim &  \int_{J_k}    \int \sum\limits_j \sum\limits_{\mathcal{R}(j)}  \left( \overline{Iu_j} Iu_{j_1} \overline{Iu_{j_2}} Iu_{j_3} \right)(t,x) \,\mathrm{d}x \mathrm{d}t  \cdot \sup\limits_{t \in J_k} N(t) \int_{|y - x(t)| \ge \frac{R}{N(t)} } \sum\limits_{j'}  \left|Iu_{j'} (t,y) \right|^2 \,\mathrm{d}y \\
\lesssim & \sup\limits_{t \in J_k}  N(t) \int_{|y - x(t)| \ge  \frac{R}{N(t)} } \sum\limits_{j' \in \mathbb{Z}}  \left|P_{\le CK} u_{j'}(t,y) \right|^2 \,\mathrm{d}y \\
\lesssim & \sup\limits_{t \in J_k}  N(t) \bigg( \int_{|y - x(t)| \ge  \frac{R}{N(t)} } \sum\limits_{j'}  \left|u_{j'}(t,y) \right|^2 \,\mathrm{d}y + \int \sum\limits_{j' \in \mathbb{Z}}  \left|P_{\ge CK} u_{j'} (t,y) \right|^2 \,\mathrm{d}y  \bigg)
\\
\lesssim & \sup\limits_{t \in J_k} N(t)  \int_{ |y - x(t)| \ge  \frac{ R}{N(t)} } \sum\limits_{j' \in \mathbb{Z}}  \left|u_{j'}(t,y) \right|^2 \,\mathrm{d}y + \sup\limits_{t \in J_k} N(t)  \int \sum\limits_{j' \in \mathbb{Z}}  \left|P_{ |\xi - \xi(t)| \ge C(R) K N(t) } u_{j'}(t,y) \right|^2 \,\mathrm{d}y
\le     o_R(1) N(J_k).
\end{align*}
Taking the summation over the characteristic interval $J_k$ of $[0,T]$, by Lemma \ref{le2.3v20}, we have
\begin{align}
& \int_0^T \iint_{|y - x(t)| \ge  \frac{R}{N(t)} }  \left|\psi_{ R N(t)^{- 1 }} '( |x- y|)  \right| N(t)  |x-y| \sum\limits_{j,j' \in \mathbb{Z}} \sum\limits_{\mathcal{R}(j)}  \left( \overline{Iu_j} Iu_{j_1} \overline{Iu_{j_2}} Iu_{j_3}  \right)(t,x)   \left|Iu_{j'}(t,y) \right|^2 \,\mathrm{d}x \mathrm{d}y \mathrm{d}t \notag \\
= & \sum\limits_{J_k \subseteq [0,T]} \int_{J_k} \iint_{|y - x(t)| \ge  \frac{ R}{N(t)} }  \left|\psi_{R N(t)^{-  1 }} '(|x-y|)  \right| N(t)
 |x - y| \notag \\
 & \quad \cdot
 \sum\limits_{j,j' \in \mathbb{Z}} \sum\limits_{\mathcal{R}(j)}  \left( \overline{Iu_j} Iu_{j_1} \overline{Iu_{j_2}} Iu_{j_3}  \right)(t,x)
 \left|Iu_{j'}(t,y) \right|^2 \,\mathrm{d}x \mathrm{d}y \mathrm{d}t  \notag \\
\lesssim & o_R(1) \sum\limits_{J_k \subseteq [0,T]} N(J_k)
\sim  o_R(1) \sum\limits_{J_k \subseteq [0,T]} \int_{J_k} N(t)^3 \,\mathrm{d}t
\sim o_R(1) \int_0^T N(t)^3 \,\mathrm{d}t = o_R(1) K. \label{eq5.25v89}
\end{align}

Thus, by \eqref{eq5.22v89}, \eqref{eq5.23v89}, and \eqref{eq5.25v89}, we have
\begin{align}
&  \int_0^T  \bigg| - \sum\limits_{j,j' \in \mathbb{Z}} \iint \psi_{RN(t)^{- 1 }}' ( |x - y|) N(t) |x - y| \sum\limits_{\mathcal{R}(j)}  \left( \overline{Iu_j} Iu_{j_1} \overline{Iu_{j_2}} Iu_{j_3} \right)(t,x) |Iu_{j'}(t,y)|^2 \,\mathrm{d}x \mathrm{d}y \bigg| \mathrm{d}t  \notag \\
\lesssim  &  R^{- \frac12 }  \left\|\vec{u} \right\|_{L_y^2 l^2 }^2 K + o_R(1) K,  \label{eq5.30v89}
\end{align}
and therefore this completes the estimate of \eqref{eq5.15v86}.

 We now turn to the estimate of \eqref{eq5.14v86}. By simple calculation, we see
\begin{align}
& - 2 \int_0^T \sum\limits_{j,j' \in \mathbb{Z}} \iint \psi_{ R N(t)^{- 1
} }  ( |x- y|) N(t) \sum\limits_{\mathcal{R}(j)} \left( \overline{Iu_j} Iu_{j_1} \overline{Iu_{j_2}}
 Iu_{j_3}  \right)(t,x) \left|Iu_{j'} (t,y) \right|^2 \,\mathrm{d}x \mathrm{d}y \mathrm{d}t  \notag \\
= & - 2  \int_0^T \sum\limits_{j,j' \in \mathbb{Z}} \iint \left( \psi_{ R N(t)^{- 1 } } (|x- y|)
  - \phi\left( \frac{ |x-y|}R \right) \right) N(t) 
 \sum\limits_{\mathcal{R}(j)} \left( \overline{Iu_j}
Iu_{j_1} \overline{Iu_{j_2}} Iu_{j_3} \right)(t,x) \left|Iu_{j'}(t,y)\right|^2 \,\mathrm{d}x \mathrm{d}y  \mathrm{d}t \notag  \\
& - 2  \int_0^T \sum\limits_{j,j' \in \mathbb{Z}} \iint \phi \left( \frac{ N(t) |x - y|}R \right) N(t)
\sum\limits_{\mathcal{R}(j)}  \left( \overline{Iu_j} Iu_{j_1} \overline{Iu_{j_2}} Iu_{j_3} \right)(t,x) \left|Iu_{j'}(t,y) \right|^2 \,\mathrm{d}x \mathrm{d}y
\mathrm{d}t \notag \\
= & 2  \int_0^T \sum\limits_{j,j' \in \mathbb{Z}} \iint |x- y|
\psi_{ R N(t)^{- 1} }' ( |x - y|) N(t)
\sum\limits_{\mathcal{R}(j)} \left( \overline{Iu_j} Iu_{j_1} \overline{Iu_{j_2}} Iu_{j_3} \right)(t,x) \left|Iu_{j'} (t,y) \right|^2 \,\mathrm{d}x \mathrm{d}y
\mathrm{d}t   \label{eq5.16v86} \\
& - 2 \int_0^T \iint \phi \left( \frac{ N(t)  |x - y|}R \right) N(t)  \sum\limits_{j,j' \in \mathbb{Z}} \sum\limits_{\mathcal{R}(j)} \left( \overline{Iu_j} Iu_{j_1} \overline{Iu_{j_2}} Iu_{j_3} \right)(t,x) \left|Iu_{j'}(t,y) \right|^2 \,\mathrm{d}x \mathrm{d}y \mathrm{d}t , \label{eq5.17v86}
\end{align}
we see the term \eqref{eq5.16v86} has been estimated previously in \eqref{eq5.15v86}, while the estimate of the term  \eqref{eq5.17v86} will be incorporated into the estimate of the term \eqref{eq5.12v86} in the following.

We now turn to 
\eqref{eq5.12v86}.
We will estimate this term together with the remainder \eqref{eq5.17v86} in the estimate of \eqref{eq5.14v86}. We consider
\begin{align*}
& 4 \sum\limits_{j,j' \in \mathbb{Z}} \iint \frac1{ 2\pi  R^2} \int \varphi \left( \left|\frac{x N(t) }R - s \right| \right) \varphi
 \left( \left|\frac{y N(t) }R - s \right| \right)  N(t)
 \left|\nabla  \left( e^{- i x \xi(s)} Iu_j(t,x) \right) \right|^2  \left|Iu_{j'}(t,y) \right|^2 \,\mathrm{d}x \mathrm{d}y  \mathrm{d}s \\
& - 2 \sum\limits_{j,j' \in \mathbb{Z}} \iint \phi \left( \frac{ N(t) |x- y|}R  \right)  N(t) \sum\limits_{\mathcal{R}(j)}  \left( \overline{Iu_j} Iu_{j_1} \overline{Iu_{j_2}} Iu_{j_3}  \right)(t,x)  \left|Iu_{j'}(t,y) \right|^2 \,\mathrm{d}x \mathrm{d}y.
\end{align*}
Let $\chi \in C_0^\infty$ such that
\begin{align*}
\chi(x) =
\begin{cases}
1, |x| \le R - 2 \sqrt{R},\\
0, |x|  \ge R - \sqrt{R},
\end{cases}
\end{align*}
we have $|\chi''(x)|\lesssim \frac1R$, $|\chi'(x)|\lesssim \frac1{\sqrt{R}}$, and $\varphi(x) \ge \chi^2 (x)$.
Thus, we have
\begin{align}
& \sum\limits_{j \in \mathbb{Z}} \int \chi^2 \left( \left|\frac{x N(t) }R - s \right| \right) \left|\nabla \left(e^{- i x \xi(s)} Iu_j(t,x) \right) \right|^2 \,\mathrm{d}x  \notag \\
= & \sum\limits_{j \in \mathbb{Z}} \int  \left|\nabla \left( \chi \left( \left|\frac{x N(t) }R - s \right| \right) e^{- i x\xi(s)} Iu_j(t,x) \right)\right|^2 \,\mathrm{d}x \notag\\
& \quad +  \frac{N(t)^2} {R^2} \sum\limits_{j \in \mathbb{Z}} \int \chi \left( \left|\frac{x N(t) }R - s \right| \right)  \left( \Delta \chi \right) \left( \left|\frac{x N(t) }R - s \right| \right) \left|Iu_j(t,x) \right|^2 \,\mathrm{d}x \notag  \\
\ge & \sum\limits_{j \in \mathbb{Z}} \int  \left|\nabla \left(\chi \left( \left|\frac{x N(t) }R - s \right| \right) e^{-i x \xi(s)} Iu_j(t,x) \right)\right|^2 \,\mathrm{d}x - \frac{c N(t)^2 }{R^2}  \left\|\vec{u} \right\|_{L_x^2 l^2}^2.
\label{eq5.18v87}
 \end{align}
We now consider
\begin{align*}
&  \frac2{ 2\pi  R^2} \int_0^T \sum\limits_{j,j' \in \mathbb{Z}} \iiint  \left( \left|\chi \left( \left|\frac{x N(t) }R - s \right| \right) \right|^4 - \varphi \left( \left|\frac{x N(t) }R - s \right| \right) \right) \varphi \left( \left|\frac{ y N(t) }R - s \right| \right) N(t) \,\mathrm{d}s \\
&  \quad \cdot  \sum\limits_{\mathcal{R}(j)} \left( \overline{Iu_j}Iu_{j_1} \overline{Iu_{j_2}} Iu_{j_3}  \right)(t,x) \left|Iu_{j'}(t,y) \right|^2
\,\mathrm{d}x \mathrm{d}y \mathrm{d}t. 
\end{align*}
Direct calculation shows for $x,y$ satisfy $|x-y| \le \frac{R^2}{4N(t)} $, 
\begin{align}\label{eq5.19v20}
\frac1{2 \pi  R^2} \int_{\mathbb{R}^2}  \left( \varphi\left( \left|\frac{N(t) x}R - s\right|\right) - \chi^4 \left( \left|\frac{N(t) x}R - s\right|\right) \right) \varphi\left( \left|\frac{N(t) y}R - s\right|\right) \,\mathrm{d}s \lesssim \frac1{\sqrt{R}}, 
\end{align}
then by \eqref{eq5.19v20} and  Lemma \ref{le2.3v20}, we have 
\begin{align}
& \int_0^T \frac2{ 2\pi  R^2} \sum\limits_{j,j' \in \mathbb{Z}} \iiint_{|x - y| \le \frac{R^2}{ 4 N(t)} }
 \left(  \left|\chi \left( \left|\frac{x N(t) }R - s \right| \right) \right|^4 -
 \varphi
 \left( \left|\frac{x N(t) }R - s \right| \right)  \right) \varphi \left( \left|\frac{y N(t) }R - s \right| \right) N(t)  \,\mathrm{d}s  \notag \\
& \quad \cdot \sum\limits_{\mathcal{R}(j)} \left( \overline{Iu_j} Iu_{j_1} \overline{Iu_{j_2}}
Iu_{j_3} \right)(t,x) \left|Iu_{j'}(t,y) \right|^2 \,\mathrm{d}x \mathrm{d}y \mathrm{d}t \notag \\
\lesssim & \frac1{\sqrt{R}} \sum\limits_{j,j' \in \mathbb{Z}} \sum\limits_{J_k \subseteq [0, T]}  N(J_k)
\int_{J_k}  \iint \sum\limits_{\mathcal{R}(j)} \left( \overline{Iu_j} Iu_{j_1}
 \overline{Iu_{j_2}} Iu_{j_3}  \right)(t,x)
 \left|Iu_{j'} (t,y) \right|^2 \,\mathrm{d}x \mathrm{d}y \mathrm{d}t \notag\\
\lesssim & \frac1{\sqrt{R}} \|\vec{u}\|_{L_y^2 l^2}^2 \int_0^T N(t)^3 \,\mathrm{d}t
\lesssim \frac1{\sqrt{R}} \|\vec{u}\|_{L_x^2 l^2}^2 K .  \label{eq5.35v91}
\end{align}

We now consider when $|x- y | \ge \frac{R^2}{ 4 N(t)} $, then
\begin{align}
&\int_0^T \frac2{ 2\pi  R^2} \sum\limits_{j,j' \in \mathbb{Z}} \iiint_{|x - y| \ge \frac{R^2}{ 4 N(t)} }
\left( \left|\chi \left( \left|\frac{x N(t) }R- s \right| \right) \right|^4 - \varphi \left( \left|\frac{x N(t) }R- s \right| \right) \right) \varphi \left( \left|\frac{y N(t) }R- s \right| \right) N(t)  \,\mathrm{d}s \notag \\
& \quad \cdot \sum\limits_{\mathcal{R}(j)} \left( \overline{Iu_j} Iu_{j_1}
\overline{Iu_{j_2}} Iu_{j_3} \right)(t,x) \left|Iu_{j'}(t,y) \right|^2 \,\mathrm{d}x \mathrm{d}y \mathrm{d}t \notag \\
\lesssim & \int_0^T \sum\limits_{j,j' \in \mathbb{Z}} \iint_{|x- y| \ge \frac{R^2}{ 4 N(t)} } \sum\limits_{\mathcal{R}(j)}
\left( \overline{Iu_j} Iu_{j_1} \overline{Iu_{j_2}} Iu_{j_3} \right)(t,x) \left|Iu_{j'}(t,y) \right|^2 \,\mathrm{d}x \mathrm{d}y \mathrm{d}t
\lesssim \frac{K}{\sqrt{R}} + o_R(1)K. \label{eq5.36v91}
\end{align}
Therefore, by \eqref{eq5.35v91} and \eqref{eq5.36v91}, we have
\begin{align}
&  \frac2{ 2\pi  R^2} \int_0^T \sum\limits_{j,j' \in \mathbb{Z}} \iiint  \left( \left|\chi \left( \left|\frac{x N(t) }R - s \right| \right) \right|^4 - \varphi \left( \left|\frac{x N(t) }R - s \right| \right) \right) \varphi \left( \left|\frac{y N(t) }R - s \right| \right) N(t) \,\mathrm{d}s   \notag   \\
& \quad \cdot  \sum\limits_{\mathcal{R}(j)} \left( \overline{Iu_j}Iu_{j_1} \overline{Iu_{j_2}} Iu_{j_3}  \right)(t,x) \left|Iu_{j'}(t,y) \right|^2
\,\mathrm{d}x \mathrm{d}y \mathrm{d}t
\lesssim o_R(1)K. \label{eq5.35v91}
\end{align}
Therefore, by \eqref{eq5.18v87} and \eqref{eq5.35v91}, we have
\begin{align*}
 & 4  \sum\limits_{j,j' \in \mathbb{Z}} \int_0^T \iint \frac1{ 2 \pi R^2} \int \varphi \left( \left |\frac{x N(t) }R - s \right| \right) \varphi \left( \left |\frac{y N(t) }R - s \right| \right)  N(t) \\
  & \qquad \cdot \left|\nabla \left(e^{-i x \xi(s)} Iu_j(t,x) \right) \right|^2 \left|Iu_{j'}(t,y) \right|^2 \, \mathrm{d}s\mathrm{d}x \mathrm{d}y \mathrm{d}t \\
 & - 2 \sum\limits_{j,j' \in \mathbb{Z}}  \int_0^T \iint \phi \left( \frac{N(t) |x- y|}R \right) N(t)
 \sum\limits_{\mathcal{R}(j)}
 \left( \overline{Iu_j} Iu_{j_1} \overline{Iu_{j_2}} Iu_{j_3} \right)(t,x) \left|Iu_{j'}(t,y) \right|^2 \,\mathrm{d}x \mathrm{d}y  \mathrm{d}t \\
 \ge & \frac2{ \pi  R^2} \sum\limits_{j,j' \in \mathbb{Z}} \int_0^T \iiint \left|\nabla \left( \chi
 \left( \left|\frac{x N(t) }R - s \right| \right) e^{-i x \xi(s)} Iu_j(t,x) \right) \right|^2 \,\mathrm{d}x \cdot
\varphi \left( \left|\frac{y N(t) }R- s \right| \right) \left|Iu_{j'}(t,y) \right|^2 N(t) \,\mathrm{d}y \mathrm{d}s \mathrm{d}t  \\
 - &   \frac{1}{\pi R^2}  \sum\limits_{j,j' \in \mathbb{Z}}   
\int_0^T \iiint  \left|\chi \left( \left|\frac{x N(t) }R - s \right| \right) \right|^4   \varphi \left( \left|\frac{y N(t) }R - s \right| \right) N(t)  \,\mathrm{d}s \\
& \qquad \cdot
 \sum\limits_{\mathcal{R}(j)}  \left( \overline{Iu_j} Iu_{j_1} \overline{Iu_{j_2}} Iu_{j_3}  \right)(t,x) \left|Iu_{j'}(t,y) \right|^2  \,\mathrm{d}x \mathrm{d}y \mathrm{d}t \\
& - C \frac{K}{   R^4}   \left\|\vec{u} \right\|_{L_x^2 l^2}^4 - o_R(1) K .
\end{align*}
Then by Theorem \ref{th3.8v32} and the fact $\left\|\vec{u}(0) \right\|_{L^2_x l^2 } < \frac1{\sqrt2} \|Q\|_{L^2}$, one can find $\eta = \eta\left( \left\|\vec{u}(0) \right\|_{L^2_x l^2}\right) > 0$ such that
\begin{align*}
& \sum\limits_{j \in \mathbb{Z}} \int  \left|\nabla  \left(
 \chi \left(  \left|\frac{x N(t) }R - s \right| \right) e^{-i x \xi(s)} Iu_j(t,x)
\right)
\right|^2 \,\mathrm{d}x \\
& \quad
- \frac12 \sum\limits_{j \in \mathbb{Z}} \int  \left|\chi \left(  \left|\frac{x N(t) }R - s \right| \right)
\right|^4 \sum\limits_{\mathcal{R}(j)}
\left( \overline{Iu_j} Iu_{j_1} \overline{Iu_{j_2}} Iu_{j_3}  \right)(t,x) \,\mathrm{d}x \\
\ge & \frac12  \Bigg( \frac{ \sqrt{\frac12} \|Q\|_{L^2} }{ \Big( \sum\limits_{j \in \mathbb{Z}}
 \left\|\chi  \left( \left|\frac{x N(t) }R - s \right|  \right) e^{-i x \xi(s)} Iu_j(t,x) \right\|_{L_x^2}^2 \Big)^\frac12 } \Bigg)^2
\cdot \sum\limits_{j \in \mathbb{Z}} \int  \left|\chi \left( \left|\frac{x N(t) }R - s \right| \right) \right|^4
\sum\limits_{\mathcal{R}(j)} \left( \overline{Iu_j} Iu_{j_1} \overline{Iu_{j_2}} Iu_{j_3}  \right)(t,x) \,\mathrm{d}x \\
& - \frac12 \sum\limits_{j \in \mathbb{Z}} \int  \left|\chi \left( \left|\frac{x N(t) }R - s \right| \right) \right|^4
\sum\limits_{\mathcal{R}(j)} \left( \overline{Iu_j} Iu_{j_1} \overline{Iu_{j_2}} Iu_{j_3}  \right)(t,x) \,\mathrm{d}x \\
\ge & \eta \sum\limits_{j \in \mathbb{Z}} \int  \left|\chi \left( \left|\frac{x N(t) }R - s \right| \right) \right|^4
\sum\limits_{\mathcal{R}(j)} \left( \overline{Iu_j} Iu_{j_1} \overline{Iu_{j_2}} Iu_{j_3}  \right)(t,x) \,\mathrm{d}x.
\end{align*}
By
\begin{align*}
\frac1{ 2\pi  R^2} \int \chi^4  \left( \left|\frac{x N(t) }R - s \right| \right) \varphi \left( \left|\frac{y N(t) }R - s \right| \right) \,\mathrm{d}s \ge \frac14,
\end{align*}
we have
\begin{align}
&\frac4{ 2\pi  R^2} \sum\limits_{j,j' \in \mathbb{Z}} 
 \iint  \left|\chi \left( \left|\frac{x N(t) }R - s \right| \right) \right|^4 \sum\limits_{\mathcal{R}(j)}
 \left( \overline{Iu_j} Iu_{j_1} \overline{Iu_{j_2}} Iu_{j_3} \right)(t,x) \,\mathrm{d}x \notag  \\
 & \qquad \cdot
 \int \varphi \left( \left|\frac{y N(t)  }R - s \right|  \right)  N(t) \left|Iu_{j'}(t,y) \right|^2 \,\mathrm{d}y \mathrm{d}s \notag  \\
\ge & c
\sum\limits_{j,j' \in \mathbb{Z}} \iint_{|x - y| \le \frac{R^2}{ 4 N(t)} } \sum\limits_{\mathcal{R}(j)}
 \left( \overline{Iu_j} Iu_{j_1} \overline{Iu_{j_2}} Iu_{j_3}  \right)(t,x) N(t)  \left|Iu_{j'}(t,y) \right|^2 \,\mathrm{d}x \mathrm{d}y  \notag \\
= & c  \sum\limits_j \int \sum\limits_{\mathcal{R}(j)} \left( \bar{u}_j u_{j_1}  \bar{u}_{j_2} {u}_{j_3}  \right)(t,x) \,\mathrm{d}x N(t)
 \left\|\vec{u} \right\|_{L_y^2 l^2}^2  \label{eq5.38v91} \\
& - c \sum\limits_{j \in \mathbb{Z}} \int \sum\limits_{\mathcal{R}(j)}
 \left( \overline{P_{\ge CK} u_j} u_{j_1} \bar{u}_{j_2} u_{j_3}  \right)(t,x) N(t) \,\mathrm{d}x
 \left\|\vec{u} \right\|_{L_y^2 l^2}^2 \label{eq5.39v91} \\
& - c  \sum\limits_{j,j' \in  \mathbb{Z}} \iint \sum\limits_{\mathcal{R}(j)}  \left( \bar{u}_j u_{j_1} \bar{u}_{j_2}
u_{j_3}  \right)(t,x) \left|P_{\ge CK} u_{j'}(t,y) \right|^2 N(t) \,\mathrm{d}x \mathrm{d}y \label{eq5.40v91} \\
& - c  \sum\limits_{j,j' \in \mathbb{Z}} \iint_{|x - y| \ge \frac{R^2}{ 4 N(t)} }
 \sum\limits_{\mathcal{R}(j)}  \left( \overline{Iu_j} Iu_{j_1} \overline{Iu_{j_2}} Iu_{j_3}  \right)(t,x)
 \left|Iu_{j'}(t,y) \right|^2 N(t) \,\mathrm{d}x \mathrm{d}y . \label{eq5.41v91}
\end{align}

We consider \eqref{eq5.38v91}. Integrating on $[0,T]$, by Lemma \ref{le2.3v20}, we obtain 
\begin{align}
& \int_0^T \sum\limits_j \int \sum\limits_{\mathcal{R}(j)} \left( \bar{u}_j u_{j_1} \bar{u}_{j_2} u_{j_3} \right)(t,x) N(t)
\,\mathrm{d}x  \left\|\vec{u} \right\|_{L_y^2 l^2}^2 \,\mathrm{d}t\notag  \\
\gtrsim &   \left\|\vec{u} \right\|_{L_y^2 l^2}^2 \sum\limits_{J_k \subseteq [0,T]}   N(J_k) 
\sim  \left\|\vec{u} \right\|_{L_y^2 l^2}^2 \sum\limits_{J_k \subseteq [0,T] } \int_{J_k}  N(t)^3 \,\mathrm{d}t \sim  \left\|\vec{u} \right\|_{L_y^2 l^2}^2 \int_0^T   N(t)^3 \,\mathrm{d}t
\sim \|\vec{u}\|_{L_y^2 l^2}^2 K. \label{eq5.42v91}
\end{align}

We now consider \eqref{eq5.40v91}. By \eqref{eq2.10v20}, we have 
\begin{align*}
\left\|P_{|\xi - \xi(t)| \ge C(\eta) N(t)} \vec{u} \right\|_{L_y^2 l^2}^2 < \eta.
\end{align*}

Since $\int_0^T |\xi'(t)| \,\mathrm{d}t << CK$, we have
\begin{align*}
\left\| P_{ \ge CK} \vec{u} \right\|_{L_y^2 l^2} \le  \left\|P_{|\xi - \xi(t)| \ge C(\eta_0) K} \vec{u}  \right\|_{L_y^2 l^2} \le \eta_0,
\end{align*}
where we take $C>> C( \eta)$. Integrating \eqref{eq5.40v91} on $[0,T]$, we obtain
\begin{align}
& \int_0^T \sum\limits_{j \in \mathbb{Z}} \int \sum\limits_{\mathcal{R}(j)}  \left( \bar{u}_j u_{j_1} \bar{u}_{j_2} u_{j_3} \right)(t,x) N(t)
 \left\|P_{\ge CK } \vec{u}(t) \right\|_{L_y^2 l^2}^2 \,\mathrm{d}t  \notag \\
\lesssim &  \eta_0 \int_0^T N(t)
\sum\limits_{j \in \mathbb{Z}} \sum\limits_{\mathcal{R}(j)} \int  \left( \bar{u}_j u_{j_1} \bar{u}_{j_2} u_{j_3}  \right)(t,x) \,\mathrm{d}x \mathrm{d}t
\sim  \eta_0 \sum\limits_{J_k \subseteq [0,T]} N(J_k) \sim \eta_0 K. \label{eq5.43v91}
\end{align}

We now consider \eqref{eq5.39v91}. Integrating on $[0,T]$, we obtain
\begin{align}
&\int_0^T \sum\limits_{j \in \mathbb{Z}} \int \sum\limits_{\mathcal{R}(j)}
 \left( \overline{P_{\ge CK} u_j} u_{j_1} \bar{u}_{j_2} u_{j_3}  \right)(t,x) N(t)
\|\vec{u}\|_{L_y^2 l^2}^2 \,\mathrm{d}x \mathrm{d}t \notag \\
\lesssim & \|\vec{u}\|_{L_y^2 l^2}^2 \int_0^T \int \sum\limits_{j \in \mathbb{Z}} \sum\limits_{\mathcal{R}(j)}
\left( \overline{P_{|\xi - \xi(t)| \ge C(\eta_0) N(t)} u_j} u_{j_1} \bar{u}_{j_2} u_{j_3}  \right)(t,x)  N(t) \,\mathrm{d}x \mathrm{d}t
\lesssim  o_{C(\eta_0)}(1) K. \label{eq5.44v91}
\end{align}
We consider \eqref{eq5.41v91}. By Lemma \ref{le2.6v20}, we have on the characteristic interval $J_k \subseteq [0,T]$,
\begin{align}\label{eq5.45v92}
\int_{J_k}  \int_{|x - x(t)|\ge \frac{R^2}{8 N(t)} }  \left\|\vec{u}(t,x) \right\|_{l^2}^4 \,\mathrm{d}x \mathrm{d}t \le o_{R^2} (1),
\end{align}
and
\begin{align}
\int_{J_k}
 \int_{|x - x(t)|\ge \frac{R^2}{8 N(t)} }  \left\| P_{\ge CK } \vec{u}(t,x) \right\|_{l^2}^4 \,\mathrm{d}x \mathrm{d}t
&  \le \int_{J_k}  \int \left\|P_{\ge CK} \vec{u}(t,x)  \right\|_{l^2}^4 \,\mathrm{d}x \mathrm{d}t \notag \\
& \lesssim \int_{J_k}  \int  \left\|P_{|\xi - \xi(t)| \ge C(\eta_0) N(t) } \vec{u}(t,x)  \right\|_{l^2}^4 \,\mathrm{d}x \mathrm{d}t
\lesssim o_{C(\eta_0)}(1), \label{eq5.46v92}
\end{align}
where we use $  \left|\xi'(t) \right|  \lesssim N(t)^3 $ in the last but one inequality.

 By \eqref{eq2.10v20}, 
  we also have
\begin{align}\label{eq5.47v92}
\int_{|y - x(t)|\ge \frac{R^2}{8 N(t)} } \left\|\vec{u}(t,y) \right\|_{l^2}^2 \,\mathrm{d}y \le o_{R^2}(1),
\end{align}
and
\begin{align} \label{eq5.48v92}
\int_{|y - x(t)| \ge \frac{R^2}{ 8 N(t)} }  \left\|P_{\ge CK} \vec{u} (t,y)  \right\|_{l^2}^2 \,\mathrm{d}y
\le \int  \left\|P_{\ge CK} \vec{u}(t,y)  \right\|_{l^2}^2 \,\mathrm{d}y
\le \int  \left\|P_{|\xi - \xi(t)| \ge C(\eta_0) N(t)} \vec{u}(t,y)  \right\|_{l^2}^2 \,\mathrm{d}y \lesssim \eta_0.
\end{align}
Integrating on $[0,T]$, by \eqref{eq5.45v92}, \eqref{eq5.46v92}, \eqref{eq5.47v92}, and \eqref{eq5.48v92}, we have
\begin{align*}
& \int_0^T \iint_{|x - y| \ge \frac{R^2}{ 4 N(t)} } \sum\limits_{j,j' \in \mathbb{Z}} \sum\limits_{\mathcal{R}(j)}
 \left( \overline{Iu_j} Iu_{j_1} \overline{Iu_{j_2}} Iu_{j_3}  \right)(t,x) \left|Iu_{j'}(t,y) \right|^2 N(t)  \,\mathrm{d}x \mathrm{d}y \mathrm{d}t \\
\lesssim & \int_0^T \int_{| x- x(t)|\ge \frac{R^2}{ 8 N(t)} } \sum\limits_j \sum\limits_{\mathcal{R}(j)}
\left( \overline{Iu_j} Iu_{j_1} \overline{Iu_{j_2}} Iu_{j_3}  \right)(t,x) N(t)  \,\mathrm{d}x \mathrm{d}t
\cdot \sup\limits_{t \in [0,T)} \int_{|y - x(t)|\ge \frac{R^2}{ 8 N(t)} } \sum\limits_{j' \in \mathbb{Z}}
\left|Iu_{j'}(t,y) \right|^2 \,\mathrm{d}y  \\
\lesssim & \int_0^T N(t)  \int  \left\|\vec{u}(t,x) \right\|_{l^2}^4  \,\mathrm{d}x \mathrm{d}t
\cdot \sup\limits_{t \in [0,T]} \int_{|y - x(t)|\ge \frac{R^2}{ 8 N(t) } }  \left(  \left\|\vec{u}(t,y) \right\|_{l^2}^2 +
\left\|P_{\ge CK} \vec{u}(t,y) \right\|_{l^2}^2  \right) \,\mathrm{d}y
\lesssim o_R(1) K.
\end{align*}
Thus, we get
\begin{align}\label{eq5.49v92}
\int_0^T \iint_{|x - y| \ge \frac{R^2}{ 4 N(t)} } \sum\limits_{j,j' \in \mathbb{Z}} \sum\limits_{\mathcal{R}(j)}
\left( \overline{Iu_j} Iu_{j_1} \overline{Iu_{j_2}} Iu_{j_3}  \right)(t,x) N(t) \left|Iu_{j'}(t,y) \right|^2
\,\mathrm{d}x \mathrm{d}y \mathrm{d}t \lesssim o_R(1) K .
\end{align}
Collecting the above estimates \eqref{eq5.12v86}-\eqref{eq5.18v89}, \eqref{eq5.25v89new}, \eqref{eq5.30v89}, \eqref{eq5.16v86},
 \eqref{eq5.17v86}, \eqref{eq5.42v91}, \eqref{eq5.43v91}, \eqref{eq5.44v91}, \eqref{eq5.49v92}, and \eqref{eq5.7v89}, we obtain
\begin{align}
& \quad \int_0^T M'(t) \,\mathrm{d}t \notag \\
\ge &  c  \eta K  - o_R(1) K
  - R^2 o(K)    \label{eq6.5v20}\\
 & + \int_0^T \sum\limits_{j,j'\in \mathbb{Z}} \iint_{\mathbb{R}^2 \times \mathbb{R}^2}  \frac{d}{dt}
 \left( \psi_{RN(t)^{- 1 } } (|x-y|) (x-y) N(t)\right) \cdot \Im \left( \overline{Iu_j(t,x)}
 \nabla_x Iu_j(t,x)\right) \left|Iu_{j'} (t,y) \right|^2
  \mathrm{d}x \mathrm{d}y \mathrm{d}t.  \label{eq6.7v20}
\end{align}

We now consider \eqref{eq6.7v20}.
We note
\begin{align*}
\frac{d}{dt} \left( \psi_{RN(t)^{-  1 }} ( |x-y|) (x-y)  N(t)\right)
= N'(t)\frac1{ 2\pi R^2} \int_{\mathbb{R}^2}  \varphi \left( \left|\frac{x N(t)} R - s\right|\right) \varphi\left( \left|\frac{y N(t)}R - s\right|\right) (x-y) \,\mathrm{d}s.
\end{align*}

Notice that
\begin{align*}
N'(t) \sum\limits_{j,j'\in \mathbb{Z}} \iint_{\mathbb{R}^2\times \mathbb{R}^2}  \varphi\left( \left| \frac{x N(t) }R - s\right|\right) \varphi\left( \left|\frac{y N(t)}R- s\right|\right) \Im \left( \overline{Iu_j (t,x)} \nabla_x Iu_j (t,x) \right) \cdot (x-y) \left|Iu_{j'} (t,y) \right|^2 \,\mathrm{d}x \mathrm{d}y
\end{align*}
is also invariant under the Galilean transformation $I\vec{u}(t,x) \mapsto e^{- i x \xi(s)} I\vec{u}(t,x)$. 
 Then for any $\epsilon > 0$, by the Cauchy-Schwarz inequality, we have
\begin{align}
& N'(t) \sum\limits_{j,j' \in \mathbb{Z}} \iint_{\mathbb{R}^2\times \mathbb{R}^2}  \varphi \left( \left|\frac{x N(t)} R- s\right|\right) \varphi\left( \left|\frac{y N(t)} R - s\right|\right) \notag \\
& \qquad \cdot  \Im \left( \overline{Iu_j} \nabla_x \left( e^{- i x\xi(s)} Iu_j \right)\right)(t,x) \cdot (x-y)  \left|Iu_{j'}(t,y) \right|^2   \,\mathrm{d}x \mathrm{d}y  \notag \\
  \le &\  \epsilon N(t) \sum\limits_{j,j'\in \mathbb{Z}} \iint_{\mathbb{R}^2\times \mathbb{R}^2}  \varphi\left( \left|\frac{x N(t)}R - s\right|\right) \varphi\left( \left|\frac{y N(t)}R - s\right|\right)  \left|\nabla_x  \left(e^{- i x \xi(s)} Iu_j \right)(t,x)\right|^2 \left|Iu_{j'} (t,y) \right|^2 \,\mathrm{d}x \mathrm{d}y \label{eq6.12v20}\\
&   + C(\epsilon) \frac{ (N'(t))^2}{N(t)^3} \sum\limits_{j,j' \in \mathbb{Z}} \iint_{\mathbb{R}^2\times \mathbb{R}^2}  \varphi\left( \left|\frac{x N(t)}R - s\right|\right) \varphi\left( \left|\frac{y N(t)}R -s\right|\right)  \left|Iu_{j} (t,x) \right|^2  \left|Iu_{j'} (t,y) \right|^2 \,\mathrm{d}x \mathrm{d}y. \label{eq6.13v20}
\end{align}
For $\epsilon >0$ small, the contribution of \eqref{eq6.12v20} will be absorbed into \eqref{eq5.12v86},
which in turn being absorbed into the first term in \eqref{eq6.5v20}.
 To estimate \eqref{eq6.13v20}, we recall that $N(t)$ satisfies \eqref{eq5.1v100new} and
\eqref{eq5.1v100}.
Thus, taking $T$ sufficiently large, by \eqref{eq5.1v100new}, \eqref{eq2.17v20}, the conservation of mass, and
 \eqref{eq5.1v100}, we have
\begin{align}\label{eq5.21v58}
\int_0^T \eqref{eq6.13v20} \,\mathrm{d}t  \le  C(\epsilon) \int_0^T  \left|N'(t) \right| \,\mathrm{d}t \le 
\frac{C(\epsilon)}m \int_0^T N(t)
\int_{\mathbb{R}^2}
\sum\limits_{j \in \mathbb{Z}} \sum\limits_{\mathcal{R}(j)} \left( \overline{I u_j} Iu_{j_1} \overline{Iu_{j_2}} I u_{j_3}  \right)(t,x) \,\mathrm{d}x \,\mathrm{d}t
\le \frac{c \eta}2 K.
\end{align}
Therefore, \eqref{eq6.5v20}, \eqref{eq6.7v20}, \eqref{eq6.12v20}, and \eqref{eq5.21v58} yield 
\begin{align*}
\int_0^T M'(t) \,\mathrm{d}t \ge \frac{c  \eta}2 K- R^2 o(K) - o_R(1)K.
\end{align*}
On the other hand, by \eqref{eq2.10v20}, we have
\begin{align*}
\sup\limits_{t \in [0, T ]} |M(t)| \lesssim R^2 o(K).
\end{align*}
Choosing $R(\eta)$ sufficiently large, by Lemma \ref{le2.6v20} and \eqref{eq2.10v20}, since we can take $T$ large enough and make $K$ be arbitrarily large, we conclude that $\vec{u} = 0$.

\end{proof}

\subsection{Exclusion of the almost periodic solution when $\int_0^\infty N(t)^3 \,\mathrm{d}t < \infty$}
In this subsection, we exclude the case when $\int_0^\infty N(t)^3 \,\mathrm{d}t < \infty$ by using the conservation of mass and energy.
\begin{theorem}\label{th7.1v20}
There is no almost periodic solution to \eqref{eq1.1} with $\int_0^\infty N(t)^3 \,\mathrm{d}t = K < \infty$ in Theorem \ref{th4.9v51}.
\end{theorem}
\begin{proof}
By Theorem \ref{th1.12v20}, we have for $0 \le s  < 3$,
\begin{align}\label{eq7.1v20}
\|\vec{ u} \|_{L_t^\infty \dot{H}_x^s l^2 ([0,\infty) \times \mathbb{R}^2)} \lesssim_{m_0 } K^s.
\end{align}
Also by \eqref{eq2.17v20}, there exists $\xi_\infty \in \mathbb{R}^2$ with $|\xi_\infty | \lesssim K$ such that
\begin{align*}
\lim\limits_{t\to \infty} \xi(t) = \xi_\infty.
\end{align*}
Taking
\begin{align*}
\vec{v}(t,x) = e^{-it |\xi_\infty|^2} e^{-ix \xi_\infty} \vec{u}(t,x+2 t\xi_\infty),
\end{align*}
then we have
\begin{align*}
\|\vec{v} \|_{\dot{H}_x^s l^2 } \lesssim K^s.
\end{align*}
By the interpolation inequality and \eqref{eq2.10v20}, we have for any $\eta >0$,
\begin{align*}
\liminf\limits_{t\to \infty} \left\|\vec{ v} (t,x)\right\|_{\dot{H}^1_x l^2 } \lesssim \liminf\limits_{t\to \infty } C(\eta) N(t)^2+ \eta^\frac12 K \lesssim \eta^\frac12 K.
\end{align*}
Therefore, by the conservation of energy, we have
\begin{align}\label{eq7.6v20}
E(\vec{ v} )= 0.
\end{align}
By \eqref{eq3.0v96}, we have
\begin{align}\label{eq5.26v32}
E(\vec{v}(t)) \ge c \| \vec{v} (t,x)\|_{\dot{H}_x^1 l^2 }^2  ,
\end{align}
where $c $ is some constant depending on $\|\vec{v}(0)\|_{L_x^2 l^2}$. Therefore, by \eqref{eq7.6v20}, \eqref{eq5.26v32}, and \eqref{eq5.27v32}, we have $\vec{v} =0$ and therefore $\vec{u}  = 0$.
\end{proof}

\section{Sketch of the proof of Theorem \ref{th1.3}. }\label{s1se3v30}

In this section, we give a sketch of the main idea of the proof of the global well-posedness and scattering of the $N-$coupled focusing cubic nonlinear Schr\"odinger system \eqref{eq6.1v83}.

The local well-posedness and stability theory of \eqref{eq6.1v83} can be established by following the argument of the nonlinear Schr\"odinger equations, see for example, \cite{T2}. The variational estimate of \eqref{eq6.1v83} is given in Section \ref{se3v30}, and we will rely on Theorem \ref{th3.4v99}. The reduction of the non-scattering to the almost periodic solution as in Theorem \ref{th4.9v51} is similar to the argument of the two dimensional cubic nonlinear Schr\"odinger equations, we refer to \cite{TVZ2}. The exclusion of the almost periodic solution is similar to the argument in Section \ref{se5v30}, but is easier because the couple of the nonlinear terms is finite, and we do not need to care the interchange of the Lebesgue norms.

\vspace{0.5cm}

\noindent \textbf{Acknowledgments.}
X. Cheng has been supported by ``the Fundamental Research Funds for the Central Universities" (No. B210202147), and he thanks the MATRIX for hosting him and some parts of this article was preparing when Xing Cheng was in MATRIX and Monash University. Z. Guo was supported by the ARC project (No. DP170101060). G. Hwang was supported by the 2019 Yeungnam University Research Grant. H. Yoon was supported by the Basic Science Research Program through the National Research Foundation of Korea (NRF) funded by the Ministry
of Science and ICT (NRF-2020R1A2C4002615).
The authors are grateful to Professor S. Kwon, S. Masaki, G. Xu, and Z. Zhao for helpful discussion on the nonlinear Schr\"odinger equations on the cylinders and also the nonlinear Schr\"odinger system.

\end{document}